\documentclass[twoside]{amsart}

\usepackage{lmodern}
\usepackage[english]{babel} % per la sillabazione.
\usepackage[latin1]{inputenc} % specifica la codifica con cui è stato scritto il codice sorgente
\usepackage{microtype}
\usepackage[mathcal]{eucal} % per scrivere le lettere quelle belle elaborate
\usepackage{amsfonts} % tutti gli ams-pacchetti servono per simboli e caratteri nelle formule
\usepackage{amsthm}
\usepackage{amssymb}
\usepackage{amsmath}
\usepackage{booktabs}
\usepackage{latexsym}
\usepackage[pdftex]{color, graphicx}
\usepackage[all]{xy} %serve per i grafici
\usepackage[pdftex]{hyperref} % attiva gli hyperlink

\hypersetup{bookmarksopen = true, bookmarksopenlevel = 1, colorlinks=true, linkcolor = blue, pdfstartview = FitR}

\theoremstyle{plain}
\newtheorem{theorem}{\sc Theorem}[section]
\newtheorem{proposition}[theorem]{\sc Proposition}

\newtheorem{lemma}[theorem]{\sc Lemma}
\newtheorem{corollary}[theorem]{\sc Corollary}
\newtheorem{conjecture}[theorem]{\sc Conjecture}

\theoremstyle{definition}
\newtheorem{definition}[theorem]{\sc Definition}

\newtheorem{example}[theorem]{\sc Example}

\theoremstyle{remark}
\newtheorem{remark}[theorem]{\sc Remark}

\newcommand{\K}{\Bbbk}

\newcommand{\brd}[1]{\mathbb{#1}} %brd sta per BOARD
\newcommand{\fun}[5]{\begin{array}{rccc} #1 \,\colon & #2 & \longrightarrow & #3 \\ %fun sta per FUNCTION
& #4 & \longmapsto & #5 \end{array}}
\newcommand{\varfun}[3]{#1 \colon #2 \to #3}

\newcommand{\quasihopf}[1]{{}_{#1}^{}\M{}_{#1}^{#1}}
\newcommand{\quasihopfmod}[1]{{}_{\bullet}^{}{#1}_{\bullet}^{\bullet}}

\newcommand{\bimodcat}[1]{{}_{#1}^{}{\M}_{#1}^{}}
\newcommand{\lmod}[1]{{}_{\bullet}^{}{#1}}

\newcommand{\cl}[1]{\overline{#1}}
\newcommand{\lfun}[5]{{#1} \colon {#2} \longrightarrow {#3} \colon {#4} \longmapsto {#5}}
\newcommand{\sfun}[5]{{#1} \colon {#2} \rightarrow {#3} \colon {#4} \mapsto {#5}}
\newcommand{\coinv}[2]{{#1}^{\mathrm{co}{#2}}}
\newcommand{\id}{\textrm{Id}}
\newcommand{\End}{\textrm{End}}

\newcommand{\M}{\mathfrak{M}} % Monoidal category of K-vector spaces
\newcommand{\I}{\brd{I}} % unit

\usepackage{fancyhdr}

\begin{document}

\title{On the Structure Theorem for quasi-Hopf bimodules}

\author{Paolo Saracco}
\address{University of Turin, Department of Mathematics ``Giuseppe Peano'', via Carlo Alberto 10, I-10123 Torino, Italy}
\email{p.saracco@unito.it}
\urladdr{sites.google.com/site/paolosaracco}
\date{\today}
\thanks{Most of this paper is included in the master degree thesis of the author, which was developed under the supervision of A. Ardizzoni and L. El Kaoutit.}

\begin{abstract}
The Structure Theorem for Hopf modules states that if a bialgebra $A$ is a Hopf algebra (i.e. it is endowed with a so-called antipode) then every Hopf module $M$ is of the form ${M}^{\mathrm{co}{A}}\otimes A$, where ${M}^{\mathrm{co}{A}}$ denotes the space of coinvariant elements in $M$. Actually, it has been shown that this result characterizes Hopf algebras: $A$ is a Hopf algebra if and only if every Hopf module $M$ can be decomposed in such a way. The main aim of this paper is to extend this characterization to the framework of quasi-bialgebras by introducing the notion of preantipode and by proving a Structure Theorem for quasi-Hopf bimodules. We will also establish the uniqueness of the preantipode and the closure of the family of quasi-bialgebras with preantipode under gauge transformation. Then, we will prove that every Hopf and quasi-Hopf algebra (i.e. a quasi-bialgebra with quasi-antipode) admits a preantipode and we will show how some previous results, as the Structure Theorem for Hopf modules, the Hausser-Nill theorem and the Bulacu-Caenepeel theorem for quasi-Hopf algebras, can be deduced from our Structure Theorem. Furthermore, we will investigate the relationship between the preantipode and the quasi-antipode and we will study a number of cases in which the two notions are equivalent: ordinary bialgebras endowed with trivial reassociator, commutative quasi-bialgebras, finite-dimensional quasi-bialgebras.
%\subjclass[1991]{Primary 16W30.}
\end{abstract}

\maketitle

\fancyhf{}
\renewcommand{\headrulewidth}{0pt}
\thispagestyle{fancy}
\lfoot{\smallskip\footnotesize The final publication is available at Springer via \href{https://doi.org/10.1007/s10485-015-9408-9}{doi.org/10.1007/s10485-015-9408-9}.}

\tableofcontents

\section{Introduction}

Let $A$ be a bialgebra in the monoidal category of $\K$-vector spaces $(\M,\otimes, \K,a,l,r)$ ($\K$ a field). It is a well-known result in the study of Hopf algebras that $A$ is a Hopf algebra if and only if every Hopf module $M$ can be decomposed as $M\cong \coinv{M}{A}\otimes A$, where $\coinv{M}{A}$ denotes the space of coinvariant elements of $M$ (cf. \cite[Theorem 15.5]{BrWi}). Categorically speaking, this means that there exists an adjunction between the category of vector spaces $\M$ and the category of Hopf modules $\M_A^A$ such that the left adjoint functor is given by
$$\lfun{L}{\M}{\M_A^A}{V}{V\otimes A}$$
and that $L$ is an equivalence of categories if and only if $A$ is Hopf. The implication in this theorem that gives the structure result for Hopf modules is commonly known as the Fundamental or Structure Theorem of Hopf modules (see \cite[Theorem 4.1.1]{Swe} and \cite[Theorem 1.9.4]{Mo}) and it has important consequences, e.g. in the study of integrals over a Hopf algebra (the proof of uniqueness and existence of integrals over a finite dimensional Hopf algebra is based on it; cf. \cite[Corollary 5.1.6]{Swe}).

In \cite{Dri}, Drinfel'd introduced a generalization of bialgebras and Hopf algebras connected to conformal field theory: quasi-bialgebras and quasi-Hopf algebras. Roughly speaking, a quasi-bialgebra $A$ is a bialgebra with a comultiplication that is coassociative just up to conjugation by an invertible element in $A\otimes A\otimes A$, called the Drinfel'd reassociator. A quasi-Hopf algebra is a quasi-bialgebra endowed with an antimultiplicative endomorphism $s$ and with two distinguished elements $\alpha$ and $\beta$ that satisfy certain properties. The triple $(s,\alpha,\beta)$ is called a quasi-antipode for $A$. Actually, Drinfel'd's definition of a quasi-bialgebra ensures that the category of left $A$-modules ${}_A^{}\M$ is still monoidal, and his definition of quasi-Hopf algebras guarantees that the category of finite-dimensional left $A$-modules is rigid.

In 1999, Hausser and Nill (cf. \cite{HN}) extended the Structure Theorem to the framework of quasi-bialgebras: as for the Hopf case, they found that there is a functor 
$$\lfun{G}{{}_A^{}\M}{\quasihopf{A}}{M}{M\otimes A}$$
and they proved that if a quasi-bialgebra admits a quasi-antipode, then there exists a proper analogue of the space of coinvariants such that every quasi-Hopf bimodule can be decomposed in the same way (i.e. such that $G$ is an equivalence). This result enabled them to provide a theory of integrals for quasi-Hopf algebras of finite dimension and to prove that the space of integrals on a finite-dimensional quasi-Hopf algebra (also called cointegrals) has dimension one (see \cite[Theorem 4.3]{HN}). In 2002, Bulacu and Caenepeel gave an alternative definition of the space of coinvariants that has the advantage of giving rise to an alternative definition of cointegrals that still makes sense in the infinite-dimensional case (see \cite[Section 3]{BC}). Moreover, even if it turned out to be isomorphic to the one of Hausser and Nill, it can be used to give a second version of the Structure Theorem.

Unfortunately, there's no evidence that the converse of these two results holds. Actually, there exists an example in the dual context, due to Schauenburg, of a dual quasi-bialgebra for which the Structure Theorem holds, but that is not a dual quasi-Hopf algebra (cfr. \cite{Sch1} and \cite[Example 4.5.1]{Sch3}).

Ardizzoni and Pavarin studied the topic in \cite{AP1} and they came to the conclusion that a proper generalization of the antipode to dual quasi-bialgebras was what they called a preantipode: a $\K$-linear map $\varfun{S}{A}{A}$ satisfying certain properties. The main aim of this paper is to fit what they got to the framework of quasi-bialgebras. Observe that, even if it may seem just dualizing, there are a number of difficulties to overcome. First of all, the dual of a dual quasi-bialgebra is not a quasi-bialgebra in general (unless we are in the finite-dimensional case). Secondly, unlikely the `dual quasi' case we don't have a pretty definition of the space of coinvariants that helps us in defining the adjunction between ${}_A^{}\M$ and $\quasihopf{A}$ by taking inspiration from the ordinary Hopf version. On the contrary, our definition of coinvariants (that follows Hausser-Nill's one) is a consequence of the Structure Theorem itself.

In details, the paper is organized as follows.

In Section \ref{section2} we recall some basic notions concerning monoidal categories, quasi-bialgebras, quasi-Hopf bimodules and we fix our notation. In particular, we construct an adjunction between ${}_A^{}\M$ and $\quasihopf{A}$.

Section \ref{section4} is devoted to the notion of preantipode. In \ref{section4.1} we prove the main result: the Structure Theorem for quasi-Hopf bimodules. It states that for a quasi-bialgebra $A$ the functor $\varfun{-\otimes A}{{}_A^{}\M}{\quasihopf{A}}$ of Hausser and Nill is an equivalence of categories if and only if $A$ admits a preantipode, if and only if there exists a $\K$-linear map $\varfun{\tau_M}{M}{M}$ for every quasi-Hopf bimodule $M$ that satisfies certain properties. In \ref{section4.2} we introduce the space of coinvariant elements for a quasi-Hopf bimodule, namely $N=\tau_M(M)$, and we show how every quasi-Hopf bimodule $M$ over a quasi-bialgebra $A$ with preantipode is of the form $N\otimes A$. Subsection \ref{section4.3} collects the results about uniqueness of the preantipode and closure of the class of quasi-bialgebras with preantipode under gauge transformation.

In Section \ref{section5} we introduce quasi-Hopf algebras in order to show how the classical results are now consequences of the theory we developed. The cornerstone of this section is Theorem \ref{th3.2.22}, which asserts that every quasi-Hopf algebra admits a preantipode. From this result we can recover the Structure Theorem for Hopf modules (Remark \ref{rem3.3.13}) and Hausser-Nill version of the Structure Theorem for quasi-Hopf bimodules (Remark \ref{rem3.3.14}) as corollaries. Moreover, we are able to give a proof of the Structure Theorem of Bulacu and Caenepeel that doesn't require bijectivity of the quasi-antipode and is not long nor technical (Remark \ref{rem3.3.15}). Unfortunately, and unlike the dual quasi case, we are not able to exhibit an explicit example of a quasi-bialgebra with preantipode that does not admit a quasi-antipode and so we cannot say with certainty that the two concepts don't coincide, though it is very likely to be so. Nevertheless, even if it will turn out that the two are equivalent, we took a step forward. Indeed, the preantipode is actually a more handy tool than the quasi-antipode. Primarily, because it is composed by a single data: the map $\varfun{S}{A}{A}$. Secondly, because it is unique (see Theorem \ref{th3.3.9}) and not just unique up to an invertible element (as the quasi-antipode is).

Even if we believe that quasi-bialgebras with preantipode are a strictly larger class of quasi-bialgebras with respect to quasi-Hopf algebras, in \ref{section5.1} we will be able to exhibit a number of cases in which the two structures are equivalent. For example: ordinary bialgebras viewed as quasi via the trivial reassociator (Proposition \ref{prop3.2.30}), commutative quasi-bialgebras (Corollary \ref{cor3.3.11}) and, last but not the least, finite-dimensional quasi-bialgebras.

In some of these cases we are able to recover explicitly the quasi-antipode from the preantipode, as we will show at the very end of Section \ref{section5}, for example when the distinguished element $\alpha$ is invertible. We will also highlight that we can do it for much of the best known examples of non-trivial quasi-Hopf algebra. Nevertheless, up to this moment, we are not able to give general guidelines to recover the quasi-antipode from the preantipode, even in the finite-dimensional case. The heart of the problem lies in the fact that Schauenburg's proof invokes the Krull-Schmidt Theorem and this is a non-constructive result. Hence, as we will see, the relation between the quasi-antipode and the preantipode hides behind an unknown isomorphism $\widetilde{\gamma}$ (cfr. proof of Theorem \ref{thSch}).

\section{Preliminaries}\label{section2}

Recall that (see \cite[Chapter XI]{Ka}) a \textbf{monoidal category} $(\mathcal{M},\otimes,\I,a,l,r)$ is a category $\mathcal{M}$ equipped with a functor $\varfun{\otimes}{\mathcal{M}\times\mathcal{M}}{\mathcal{M}}$ (called \textbf{tensor product}) and with a distinguished object $\I$ (called the \textbf{unit}) such that $\otimes$ is associative `up to' a natural isomorphism $a$, $\I$ is a left and right unit for $\otimes$ `up to' natural isomorphisms $l$ and $r$ respectively and `all' diagrams involving $a$, $l$ and $r$ commute. Formally, this means that we have three natural isomorphisms:
\begin{gather*}
\varfun{a}{\otimes(\otimes \times \id_\mathcal{M})}{\otimes\,(\id_\mathcal{M}\times \otimes)} \quad\mathrm{\textbf{associativity constraint}} \\
\varfun{l}{\otimes(\I\times \id_\mathcal{M})}{\id_\mathcal{M}} \quad\mathrm{\textbf{left unit constraint}}\\
\varfun{r}{\otimes(\id_\mathcal{M}\times \I)}{\id_\mathcal{M}} \quad\mathrm{\textbf{right unit constraint}}
\end{gather*}
that satisfy the \textbf{Pentagon Axiom} and the \textbf{Triangle Axiom}:
\begin{equation*}
\xymatrix @!0 @C=65pt @R=40pt{\ar@{}[ddrrrr]|{\circlearrowleft}
 & ((M\otimes N)\otimes P)\otimes Q \ar[rr]^-{a_{M,N,P}\otimes Q} \ar[dl]|{a_{M\otimes N,P,Q}} & & (M\otimes (N\otimes P))\otimes Q \ar[dr]|{a_{M,N\otimes P,Q}} & \\
(M\otimes N)\otimes (P\otimes Q) \ar[drr]|{a_{M,N,P\otimes Q}} & & & & M\otimes ((N\otimes P)\otimes Q) \ar[dll]|{M\otimes a_{N,P,Q}} \\
 & & M\otimes (N\otimes (P\otimes Q)) & & 
}
\end{equation*}
\begin{equation*}
\xymatrix @!0 @=55pt{
(M\otimes \I)\otimes N \ar[rr]^-{a_{M,\I,N}} \ar[dr]_-{r_M\otimes N} \ar@{}[drr]|{\circlearrowleft} & & M\otimes (\I\otimes N) \ar[dl]^-{M\otimes l_N} \\
 & M\otimes N & 
}
\end{equation*}
for all $M$, $N$, $P$, $Q$ in $\mathcal{M}$.

The notions of algebra, module over an algebra, coalgebra and comodule over a coalgebra can be introduced in the general setting of monoidal categories. Note that we will always request algebras to be associative and unital and coalgebras to be coassociative and counital.

Let $\K$ be a fixed field and denote by $\M$ the category of $\K$-vector spaces. This is a monoidal category with tensor product and unit given by $\otimes_{{}_\K}$ and $\K$ respectively. The associativity and unit constraints are the obvious maps. Henceforth, all vector spaces, (co)algebras and linear maps are understood to be over $\K$. The unadorned tensor product $\otimes$ will denote the tensor product over $\K$, if not stated differently. In order to deal with the comultiplication and the coaction, we use the following variation of \textbf{Sweedler's Sigma Notation} (cf. \cite[Sec. 1.2]{Swe}):
$$\Delta(x):=x_1\otimes x_2 \qquad \mathrm{and} \qquad \rho(n):=n_0\otimes n_1$$
for all $x\in C$, $C$ coalgebra, and for all $n\in N$, $N$ $C$-comodule (summation understood).

A \textbf{quasi-bialgebra}, as introduced by Drinfel'd in \cite{Dri}, is a sextuple $(A,m,u,\Delta,\varepsilon,\Phi)$ where:
\begin{itemize}
\item $(A,m,u)$ is an associative unital $\K$-algebra,
\item $\varfun{\Delta}{A}{A\otimes A}$ and $\varfun{\varepsilon}{A}{\K}$ are algebra maps, called \textbf{comultiplication} and \textbf{counit},
\item $\Phi\in A\otimes A\otimes A$ is an invertible element (the \textbf{Drinfel'd reassociator} or simply the \textbf{reassociator}) such that:
\begin{equation}\label{eq:Q1}
(A\otimes A\otimes \Delta)(\Phi)\cdot(\Delta\otimes A\otimes A)(\Phi)=(1\otimes \Phi)\cdot(A\otimes \Delta\otimes A)(\Phi)\cdot(\Phi\otimes 1)
\end{equation}
\begin{equation}\label{eq:Q2}
(\varepsilon\otimes A\otimes A)(\Phi)=1\otimes 1=(A\otimes \varepsilon\otimes A)(\Phi)=1\otimes 1=(A\otimes A\otimes \varepsilon)(\Phi)
\end{equation}
\item $\Delta$ is counital with counit $\varepsilon$ and it is quasi-coassociative, meaning that the following relations hold for all $a\in A$:
\begin{equation}\label{eq:Q3}
(A\otimes\Delta)(\Delta(a))\cdot\Phi=\Phi\cdot(\Delta\otimes A)(\Delta(a))
\end{equation}
\begin{equation}\label{eq:Q4}
l_A(\varepsilon\otimes A)(\Delta(a)) = a
\end{equation}
\begin{equation}\label{eq:Q5}
r_A(A\otimes\varepsilon)(\Delta(a)) = a
\end{equation}
\end{itemize}

Hereafter we will usually omit the constraints, in order to lighten the notation.

Let $(A,m,u,\Delta,\varepsilon,\Phi)$ be a quasi-bialgebra. A \textbf{gauge transformation} (or \emph{twist}; cf. \cite{Dri}) on $A$ is an invertible element $F$ of $A\otimes A$ such that
\begin{equation*}
(A\otimes \varepsilon)(F)=1=(\varepsilon\otimes A)(F).
\end{equation*}
Given a gauge transformation, it is possible to \textbf{twist} the quasi-bialgebra $A$ via $F$ by considering $A_F:=(A,m,u,\Delta_F,\varepsilon,\Phi_F)$ where, for all $a\in A$
\begin{gather*}
\Delta_F(a):=F\cdot\Delta(a)\cdot F^{-1}, \\
\Phi_F:=(1\otimes F)\cdot(A\otimes \Delta)(F)\cdot\Phi\cdot(\Delta\otimes A)(F^{-1})\cdot(F^{-1}\otimes 1).
\end{gather*}
This is still a quasi-bialgebra (cf. \cite[Remark, p. 1422]{Dri}), i.e. quasi-bialgebras form a class closed under gauge transformation.

Moreover, the axioms of quasi-bialgebra are necessary and sufficient to state that the category ${}_A^{}\M$ of left modules over a quasi-bialgebra $A$ is monoidal in the following way. Given a left $A$-module $M$, we denote by $\sfun{\mu =\mu _{M}^{l}}{A\otimes M}{M}{a\otimes m}{a\cdot m}$, its left $A$-action. The tensor product of two left $A$-modules $M$ and $N$ is a left $A$-module itself via the diagonal action i.e. $a\cdot(m\otimes n) =(a_{1}\cdot m)\otimes (a_{2}\cdot n)$. The unit is $\K$, which is regarded as a left $A$-module via the trivial action, i.e. $a\cdot k =\varepsilon(a)k$. The associativity and unit constraints are defined, for all $M,N,P\in {}_A^{}\M$ and $m\in M,n\in N,p\in P,k\in \K ,$ by
\begin{gather*}
{({}_Aa})_{M,N,P}((m\otimes n)\otimes p):=\Phi\cdot (m\otimes (n\otimes p)), \\
l_M(k\otimes m):=km\qquad \text{and}\qquad r_M(m\otimes k):=mk.
\end{gather*}

We will denote this category by $({}_A^{}\M,\otimes, \K,{}_Aa,l,r)$. Analogously, we can define the monoidal categories $(\M{}_A^{},\otimes ,\K ,a_A,l,r)$ and $(\bimodcat{A},\otimes ,\K ,{_{A}}a{_{A}},l,r)$,
where:
\begin{gather*}
{(a_A)}_{M,N,P}((m\otimes n)\otimes p):=(m\otimes (n\otimes p))\cdot\Phi^{-1}, \\
({_{A}}a{_{A}})_{M,N,P}((m\otimes n)\otimes p):=\Phi\cdot(m\otimes (n\otimes p))\cdot\Phi^{-1}.
\end{gather*}

\begin{remark}(\cite[Section 3]{HN})
Recall that we need an analogue of the notion of Hopf module for quasi-bialgebras. Unfortunately, a quasi-bialgebra $A$ is not a coalgebra in $\M$ and, in general, neither in ${}_A^{}\M$ nor in $\M_A^{}$. However, $(A,m,m)$ as an $(A,A)$-bimodule, endowed with $\Delta$ and $\varepsilon$, becomes a coalgebra in ${}_A^{}\M{}_A^{}$ and so we can consider the category of the so-called \textbf{(right) quasi-Hopf bimodules}
$$\quasihopf{A}:=\left({}_A\M_A\right)^A.$$
The compatibility conditions with comultiplication and counit that a coaction should satisfy rewrites:
\begin{gather}
m_0\,\varepsilon(m_1)=m, \label{eq:QH1}\\
(m_0\otimes (m_1)_1\otimes (m_1)_2)\cdot\Phi=\Phi\cdot((m_0)_0\otimes (m_0)_1\otimes m_1) \label{eq:QH2}.
\end{gather}
for all $m\in M$, $M\in\quasihopf{A}$ (cf. \cite[Definition 3.1]{HN}).
\end{remark}

\subsection{An adjunction between \texorpdfstring{$\quasihopf{A}$}{} and \texorpdfstring{${}_A\M$}{}.}\label{section2.1}

Here we construct the analogue for quasi-Hopf bimodules of the adjunction between the category of vector spaces and the category of Hopf modules (for details about the classical \lq Hopf\rq\, result, cf. \cite[Theorem 15.5]{BrWi}).

\begin{lemma}(see also \cite[Section 2]{Sch2})\label{lemma2.2}
Let $(C,\Delta,\varepsilon)$ be a coalgebra in a monoidal category $(\mathcal{M},\otimes,\I,a,l,r)$ and let $\mathcal{M}^C$ be the category of right $C$-comodules and $(N,\rho_N)$ a $C$-comodule. The assignments
$$M \mapsto \left(M\otimes N,{(a_{M,N,C})}^{-1}\circ(M\otimes \rho_N)\right)\qquad \mathrm{and}\qquad f\mapsto f\otimes N$$
define a functor $\varfun{T_N}{\mathcal{M}}{\mathcal{M}^C}$. Moreover, if $N=C$, then $T_C$ is right adjoint to $\varfun{U}{\mathcal{M}^C}{\mathcal{M}}$, the underlying functor.
\end{lemma}

Therefore we have an adjunction $(U,T)$ between $\quasihopf{A}$ and ${}_A^{}\M_A^{}$ with right adjoint given by:
$$\fun{T}{{}_A\M_A}{\quasihopf{A}}{{}_\bullet^{} M_\bullet^{}}{{}_\bullet^{} M_\bullet^{}\otimes \quasihopfmod{A}}$$
where the full dots denote the given actions and coaction. Explicitly, if we denote by
$$\varfun{\rho_{M\otimes A}}{M\otimes A}{(M\otimes A)\otimes A}$$
the coaction on $M\otimes A$, we have:
\begin{subequations}
\begin{gather}
a\cdot(m\otimes b)=a_1\cdot m\otimes a_2b \\
(m\otimes b)\cdot a=m\cdot a_1\otimes ba_2 \\
\rho_{M\otimes A}(m\otimes a)=\Phi^{-1}\cdot((m\otimes a_1)\otimes a_2)\cdot\Phi \label{eq:8}
\end{gather}
\end{subequations}
for all $a,b\in A$ and $m\in M$.

Let $M$ be an $(A,A)$-bimodule and consider the left $A$-module
$$\overline{M}:=\frac{M}{MA^+},$$
where $A^+:=\ker(\varepsilon)$ is the augmentation ideal of $A$. The assignment $\sfun{L}{{}_A^{}\M{}_A^{}}{{}_A^{}\M}{M}{\overline{M}}$ defines a functor that is left adjoint to the functor $\sfun{R}{{}_A^{}\M}{{}_A^{}\M{}_A^{}}{{}_\bullet N}{{}_\bullet N_\circ}$, where the empty dot denotes the trivial action (i.e. $n\cdot a:=n\,\varepsilon(a)$ for all $n\in N$, $a\in A$). Define
\begin{equation*}
\varfun{F:=LU}{\quasihopf{A}}{{}_A^{}\M} \quad\mathrm{and}\quad \varfun{G:=TR}{{}_A^{}\M}{\quasihopf{A}}
\end{equation*}
so that for any $M\in\quasihopf{A}$ and $N\in {}_A^{}\M$ we have that $F(\quasihopfmod{M})=\lmod{\overline{M}}$ with left action given by
\begin{equation*}
a\cdot \cl{m}=\cl{a\cdot m}
\end{equation*}
for all $a\in A$, $m\in M$, and $G(\lmod{N})={}_\bullet^{}N{}_\circ^{}\otimes \quasihopfmod{A}$ with left action, right action and coaction given by
\begin{subequations}
\begin{gather}
a\cdot(n\otimes b)=a_1\cdot n\otimes a_2b, \label{eq:7a} \\
(n\otimes b)\cdot a=n\otimes ba, \label{eq:7b} \\
\rho(n\otimes b)=\Phi^{-1}\cdot ((n\otimes b_1)\otimes b_2), \label{eq:7c}
\end{gather}
\end{subequations}
for all $a,b\in A$, $n\in N$, respectively. Note also that since for any $a$ in $A$ we have that $a-\varepsilon(a)1_A\in A^+$, if $M$ is an $(A,A)$-bimodule then in the quotient $\frac{M}{MA^+}=\overline{M}$ the following relation holds for every $m\in M, a\in A$:
\begin{equation*}
\cl{m\cdot a}=\cl{m}\,\varepsilon(a).
\end{equation*}

This fact will be used frequently in the sequel without explicit mention.

\begin{remark}
Denote by $\mu_{N\otimes A}$ the \emph{right} $A$-action defined in \eqref{eq:7b}. It is easy to see that
\begin{equation}\label{eq:8bis}
\mu_{N\otimes A}\circ (N\otimes u\otimes A)\circ \left(r^{-1}_N\otimes A\right)=\id_{N\otimes A}
\end{equation}
for any left $A$-module $N$.
This will be useful in proving Theorem \ref{Th2.4}.
\end{remark}

The content of the subsequent theorem is essentially the same of \cite[Proposition 3.6]{Sch2}. 

\begin{theorem}\label{thAdj}
The functor $\sfun{F}{\quasihopf{A}}{{}_A^{}\M}{M}{\overline{M}}$ defined above is left adjoint to the functor $\sfun{G}{{}_A^{}\M}{\quasihopf{A}}{{}_\bullet^{}N}{{}_\bullet^{}N{}_\circ^{}\otimes \quasihopfmod{A}}$. Moreover, the unit $\eta$ and counit $\epsilon$ of this adjunction are given by:
\begin{gather}
\sfun{\eta_M}{M}{\overline{M}\otimes A}{m}{\cl{m_0}\otimes m_1}, \label{eq:unit} \\
\sfun{\epsilon_N}{\overline{N\otimes A}}{N}{\cl{n\otimes a}}{n\,\varepsilon(a)} \label{eq:counit}
\end{gather}
for all $M\in \quasihopf{A}$, $N\in {}_A\M$, and $\epsilon$ is always a natural isomorphism. 
\end{theorem}

\section{The preantipode for quasi-bialgebras}\label{section4}

This section is devoted to introduce the notion of preantipode and its properties. We start by showing that to be an equivalence for the adjunction $(F,G,\eta,\epsilon)$ of Theorem \ref{thAdj} is equivalent to the existence of a suitable map $\tilde{\tau}_M$ for every quasi-Hopf module $M$. After this, we investigate the relationship between the preantipode and the maps $\tilde{\tau}_M$. The main result of this section is Theorem \ref{FundStructTheo}, where we show how the preantipode is a proper analogue of the antipode for quasi-bialgebras.

\begin{theorem}\label{Th2.4}\emph{(Dual to \cite[Proposition 3.3]{AP1})}
Let $(A,m,u,\Delta,\varepsilon,\Phi)$ be a quasi-bialgebra. The following assertions are equivalent:
\begin{itemize}
\item[(i)] The adjunction $(F,G,\eta,\epsilon)$ is an equivalence of categories.
\item[(ii)] For each $M\in{}_A^{}\M{}_A^A$, there exists a $\K$-linear map $\varfun{\tilde{\tau}_M}{\cl{M}}{M}$ such that, for all $m\in M$:
\begin{gather}
\tilde{\tau}_M(\cl{m_0})\cdot m_1=m, \label{eq:9} \\
\cl{\tilde{\tau}_M(\cl{m})_0}\otimes \tilde{\tau}_M(\cl{m})_1=\cl{m}\otimes 1. \label{eq:10}
\end{gather}
\end{itemize}
\end{theorem}

For the sake of brevity, we will often omit the subscript $M$ and we will denote $\tilde{\tau}_M$ just by $\tilde{\tau}$, when there isn't the risk of confusion.

\begin{proof}
Let us begin by observing that \eqref{eq:9} and \eqref{eq:10} can be rewritten as:
\begin{equation*}
\mu_M\circ(\tilde{\tau}\otimes A)\circ \eta_M = \id_M \qquad \mathrm{and} \qquad \eta_M\circ \tilde{\tau} = (\overline{M}\otimes u)\circ r_{\overline{M}}^{-1}
\end{equation*}
respectively, where $\mu_M$ denotes the \emph{right} $A$-action on $M$, and these suggest us how to define $\tilde{\tau}$ having $\eta_M^{-1}$ or viceversa. \\
$(i)\Rightarrow(ii)$. By hypothesis $(F,G,\eta,\epsilon)$ is an equivalence, so that $\eta$ is a natural isomorphism. For each $M\in \quasihopf{A}$ define
\begin{equation}\label{eq:tauproof}
\tilde{\tau}=\eta_M^{-1}\circ (\overline{M}\otimes u)\circ r_{\overline{M}}^{-1},
\end{equation}
i.e. $\tilde{\tau}(\cl{m})=\eta_M^{-1}(\cl{m}\otimes 1)$ for all $m\in M$. It is clear that \eqref{eq:10} holds, so let us verify \eqref{eq:9}. Since $\eta$ is a morphism of quasi-Hopf bimodules, the same holds true for $\eta^{-1}$. Therefore
\begin{displaymath}
\xymatrix @C=40pt{\ar@{}[dr]|{\circlearrowleft}
(\overline{M}\otimes A)\otimes A \ar[r]^-{\eta_M^{-1}\otimes A} \ar[d]_-{\mu_{\overline{M}\otimes A}} & M\otimes A \ar[d]^-{\mu_M} \\
\overline{M}\otimes A \ar[r]_-{\eta_M^{-1}} & M
}
\end{displaymath}
commutes, i.e. $\mu_M\circ (\eta_M^{-1}\otimes A)=\eta_M^{-1}\circ \mu_{\overline{M}\otimes A}$, and
\begin{displaymath}
\begin{split}
\mu_M\circ (\tilde{\tau}\otimes A)\circ\eta_M & \stackrel{\eqref{eq:tauproof}}{=} \mu_M\circ\left(\eta_M^{-1}\otimes A\right)\circ\left(\overline{M}\otimes u\otimes A\right)\circ \left(r_{\overline{M}}^{-1}\otimes A\right)\circ \eta_M \\
  & \stackrel{\phantom{(15)}}{=} \eta_M^{-1}\circ \mu_{\overline{M}\otimes A}\circ\left(\overline{M}\otimes u\otimes A\right)\circ \left(r_{\overline{M}}^{-1}\otimes A\right)\circ \eta_M \\
  & \stackrel{\eqref{eq:8bis}}{=} \eta_M^{-1}\circ\eta_M=\id_M.
\end{split}
\end{displaymath}

$(ii)\Rightarrow (i)$. Assume that $\varfun{\tilde{\tau}}{\overline{M}}{M}$ exists for any quasi-Hopf bimodule $M$ and define
\begin{equation}\label{eq:etainv}
\lambda_M=\mu_M\circ(\tilde{\tau}\otimes A).
\end{equation}
From \eqref{eq:9} we deduce that $\lambda_M\circ \eta_M=\id_M$. On the other hand, a direct check shows that:
\begin{displaymath}
\begin{split}
\eta_M\circ\lambda_M & =\eta_M\circ \mu_M\circ (\tilde{\tau}\otimes A) \stackrel{\phantom{(16)}}{=}\mu_{\overline{M}\otimes A}\circ (\eta_M\otimes A)\circ(\tilde{\tau}\otimes A) \\
 & \stackrel{\eqref{eq:10}}{=} \mu_{\overline{M}\otimes A}\circ(\overline{M}\otimes u\otimes A)\circ (r_{\overline{M}}^{-1}\otimes A) \stackrel{\eqref{eq:8bis}}{=} \id_{\overline{M}\otimes A}
\end{split}
\end{displaymath}
for each $M\in \quasihopf{A}$ and therefore $\eta$ is a natural isomorphism.
\end{proof}

\subsection{The preantipode and the Structure Theorem}\label{section4.1}

\begin{definition}\label{defpreantipode}
A \textbf{preantipode} for a quasi-bialgebra $(A,m,u,\Delta,\varepsilon,\Phi)$ is a $\K$-linear map $\varfun{S}{A}{A}$ that satisfies:
\begin{gather}
a_1S(ba_2)=\varepsilon(a)S(b), \label{eq:P1} \\
S(a_1b)a_2=\varepsilon(a)S(b), \label{eq:P2} \\
\Phi^1S(\Phi^2)\Phi^3=1, \label{eq:P3}
\end{gather}
for all $a,b\in A$, where $\Phi^1\otimes \Phi^2\otimes \Phi^3=\Phi$ (summation understood).
\end{definition}

\begin{remark}
Note that, evaluating \eqref{eq:P1} and \eqref{eq:P2} at $b=1$, we have that
\begin{equation}\label{eq:quasiconvinv}
(\id_A*S)(a)=\varepsilon(a)S(1)=(S*\id_A)(a)
\end{equation}
for all $a\in A$, where $*$ denotes the \textbf{convolution product}, i.e. $f*g=m\circ(f\otimes g)\circ \Delta$ for all $f,g\in\End(A)$. This means that a preantipode is very close to be the convolution inverse of the identity in $\End(A)$. Moreover, applying $\varepsilon$ on both sides of \eqref{eq:P3} we get that $\varepsilon\left(S(1)\right)=1$ and applying $\varepsilon$ again on both sides of the left hand equality in \eqref{eq:quasiconvinv}, we obtain that the preantipode preserves the counit:
\begin{equation}\label{eq:15}
\varepsilon\circ S=\varepsilon.
\end{equation}
\end{remark}

The following proposition is the natural generalization of \cite[Proposition 3.4]{HN} to quasi-bialgebras with preantipode.

\begin{proposition}\label{proptau}
Let $(A,m,u,\Delta,\varepsilon,\Phi,S)$ be a quasi-bialgebra with preantipode $S$ and $M\in\quasihopf{A}$. For all $a\in A$, $m\in M$, define
\begin{gather}
\tau:M\rightarrow M:m\mapsto \Phi^1\cdot m_0\cdot S(\Phi^2 m_1)\Phi^3 \label{eq:tau}, \\
a\blacktriangleright m:=\tau(a\cdot m). \label{eq:17}
\end{gather}
Then, for all $a,b \in A$, $m\in M$, they satisfy:
\begin{subequations}
\begin{gather}
\tau(m\cdot a)=\tau(m)\,\varepsilon(a), \label{eq:taua} \\
\tau^2=\tau,  \label{eq:taub} \\
a\blacktriangleright \tau(m)=\tau(a\cdot m),  \label{eq:tauc} \\
a\blacktriangleright(b\blacktriangleright m)=(ab)\blacktriangleright m,  \label{eq:taud} \\
a\cdot \tau(m)=\tau(a_1\cdot m)\cdot a_2=(a_1\blacktriangleright \tau(m))\cdot a_2,  \label{eq:taue} \\
\tau(m_0)\cdot m_1=m,  \label{eq:tauf} \\
\tau(\tau(m)_0)\otimes \tau(m)_1=\tau(m)\otimes 1.  \label{eq:taug}
\end{gather}
\end{subequations}
\end{proposition}

\begin{proof}
Property \eqref{eq:taua} is quite easy to prove, indeed:
\begin{displaymath}
\tau(m\cdot a) = \Phi^1\cdot m_0\cdot a_1S(\Phi^2 m_1a_2)\Phi^3 \stackrel{\eqref{eq:P1}}{=}\Phi^1\cdot m_0\cdot \varepsilon(a)\,S(\Phi^2 m_1)\Phi^3 =\tau(m)\,\varepsilon(a).
\end{displaymath}
To prove $\eqref{eq:tauc}$ one uses $\eqref{eq:taua}$ to compute:
\begin{displaymath}
\begin{split}
\tau(a\cdot\tau(m)) & \stackrel{\phantom{(21)}}{=} \tau(a\Phi^1\cdot m_0\cdot S(\Phi^2m_1)\Phi^3) \stackrel{\eqref{eq:taua}}{=} \tau(a\Phi^1\cdot m_0)\, \varepsilon(S(\Phi^2m_1))\varepsilon(\Phi^3) \\
 & \stackrel{\eqref{eq:15}}{=} \tau(a\Phi^1\cdot m_0)\,\varepsilon(\Phi^2m_1)\varepsilon(\Phi^3) = \tau(a\cdot m_0)\varepsilon(m_1) =\tau(a\cdot m),
\end{split}
\end{displaymath}
for all $a\in A$ and $m\in M$. Now, $\eqref{eq:taub}$ is just $\eqref{eq:tauc}$ with $a=1$ and $\eqref{eq:taud}$ follows directly from $\eqref{eq:tauc}$ since for every $a,b\in A$ and $m\in M$
$$a\blacktriangleright(b\blacktriangleright m)=a\blacktriangleright \tau(b\cdot m)=\tau(ab\cdot m)=(ab)\blacktriangleright m.$$
Statement $\eqref{eq:taue}$ is a consequence of the quasi-coassociativity of $\Delta$:
\begin{displaymath}
\begin{split}
\tau(a_1\cdot m)\cdot a_2 & \stackrel{\phantom{(18)}}{=} \Phi^1(a_1)_1\cdot m_0\cdot S(\Phi^2(a_1)_2m_1)\Phi^3a_2 \\
 & \hspace{1.9pt}\stackrel{\eqref{eq:Q3}}{=}\hspace{1.9pt} a_1\Phi^1\cdot m_0\cdot S((a_2)_1\Phi^2m_1)(a_2)_2\Phi^3 \\
 & \stackrel{\eqref{eq:P2}}{=} a_1\Phi^1\cdot m_0\cdot S(\Phi^2m_1)\Phi^3\varepsilon(a_2) = a\cdot \tau(m)
\end{split}
\end{displaymath}
for all $a\in A$ and $m\in M$. Furthermore, $\eqref{eq:tauf}$ follows from the fact that for all $m\in M$
\begin{displaymath}
\begin{split}
\tau(m_0)\cdot m_1 & \stackrel{\phantom{(18)}}{=} \Phi^1\cdot (m_0)_0\cdot S(\Phi^2 (m_0)_1)\Phi^3 m_1 \stackrel{\eqref{eq:QH2}}{=} m_0\cdot \Phi^1S((m_1)_1\Phi^2) (m_1)_2\Phi^3 \\
 & \stackrel{\eqref{eq:P2}}{=} m_0\cdot \Phi^1S(\Phi^2)\varepsilon(m_1)\Phi^3 \stackrel{\eqref{eq:P3}}{=} m.
\end{split}
\end{displaymath}

Finally, let us verify $\eqref{eq:taug}$. \emph{In the calculations that follows we are going to indicate with $\Psi=\Psi^1\otimes \Psi^2\otimes \Psi^3$ another copy of $\Phi$.} For all $m\in M$
\begin{displaymath}
\begin{split}
\tau(\tau(m)_0) & \otimes \tau(m)_1 = \tau\left(\left(\Phi^1\cdot m_0\cdot S(\Phi^2 m_1)\Phi^3\right)_0\right)\otimes\left(\Phi^1\cdot m_0\cdot S(\Phi^2 m_1)\Phi^3\right)_1 \\
 & \stackrel{\eqref{eq:taua}}{=} \tau\left(\left(\Phi^1\cdot m_0\right)_0\right)\, \varepsilon\left(\left(S(\Phi^2 m_1)\Phi^3\right)_1\right)\otimes\left(\Phi^1\cdot m_0\right)_1\left( S(\Phi^2 m_1)\Phi^3\right)_2 \\
 & \stackrel{\phantom{(24a)}}{=} \tau\left((\Phi^1)_1\cdot (m_0)_0\right)\otimes(\Phi^1)_2 (m_0)_1S(\Phi^2 m_1)\Phi^3\\
 & \hspace{3.2pt}\stackrel{\eqref{eq:QH2}}{=} \tau\left((\Phi^1)_1\phi^1\cdot m_0\cdot \Psi^1\right)\otimes(\Phi^1)_2\phi^2 (m_1)_1\Psi^2S(\Phi^2\phi^3 (m_1)_2\Psi^3)\Phi^3\\
 & \stackrel{\eqref{eq:taua}}{=} \tau\left((\Phi^1)_1\phi^1\cdot m_0\right)\,\varepsilon\left(\Psi^1\right)\otimes(\Phi^1)_2\phi^2 (m_1)_1\Psi^2S(\Phi^2\phi^3 (m_1)_2\Psi^3)\Phi^3 \\
 & \hspace{3.2pt}\stackrel{\eqref{eq:Q2}}{=} \tau\left((\Phi^1)_1\phi^1\cdot m_0\right)\otimes(\Phi^1)_2\phi^2 (m_1)_1S(\Phi^2\phi^3 (m_1)_2)\Phi^3  \\
 & \hspace{1.6pt}\stackrel{\eqref{eq:P1}}{=} \tau\left((\Phi^1)_1\phi^1\cdot m_0\right)\otimes(\Phi^1)_2\phi^2 \varepsilon(m_1)S(\Phi^2\phi^3)\Phi^3  \\
 & \hspace{3.2pt}\stackrel{(*)}{=} \tau\left(\phi^1\Psi^1\cdot m\right)\otimes\phi^2\Phi^1\left(\Psi^2\right)_1 S\left(\left(\phi^3\right)_1\Phi^2\left(\Psi^2\right)_2\right)\left(\phi^3\right)_2\Phi^3\Psi^3 \\
 & \hspace{1.6pt}\stackrel{\eqref{eq:P1}}{=} \tau\left(\phi^1\Psi^1\cdot m\right)\otimes\phi^2\Phi^1\,\varepsilon\left(\Psi^2\right) S\left(\left(\phi^3\right)_1\Phi^2\right)\left(\phi^3\right)_2\Phi^3\Psi^3 \\
 & \hspace{1.6pt}\stackrel{\eqref{eq:P2}}{=} \tau\left(\phi^1\cdot m\right)\otimes\phi^2\Phi^1 S\left(\Phi^2\right)\,\varepsilon\left(\phi^3\right)\Phi^3 \\
 & \hspace{3.2pt}\stackrel{\eqref{eq:Q2}}{=} \tau(m)\otimes\Phi^1 S\left(\Phi^2\right)\Phi^3 \stackrel{\eqref{eq:P3}}{=}\tau(m)\otimes 1
\end{split}
\end{displaymath}
where in $(*)$ we used \eqref{eq:Q1}:
$(\Delta\otimes A\otimes A)(\Phi)\cdot(\Phi^{-1}\otimes 1)=(A\otimes A\otimes \Delta)(\Phi^{-1})\cdot(1\otimes \Phi)\cdot(A\otimes\Delta\otimes A)(\Psi).$
\end{proof}

\begin{remark}
Actually, it can be proven that $\eqref{eq:taub}$, $\eqref{eq:tauc}$, $\eqref{eq:taud}$ and $\eqref{eq:taue}$ are consequences of $\eqref{eq:taua}$, $\eqref{eq:tauf}$ and $\eqref{eq:taug}$. Indeed, we already know that $\eqref{eq:taub}$ and $\eqref{eq:taud}$ follows directly from $\eqref{eq:tauc}$. Moreover, $\eqref{eq:taua}$ and $\eqref{eq:tauf}$ together imply that for all $a\in A$, $m\in M$,
$$\tau(a\cdot m)\stackrel{\eqref{eq:tauf}}{=}\tau(a\cdot \tau(m_0)\cdot m_1)\stackrel{\eqref{eq:taua}}{=}\tau(a\cdot\tau(m_0))\varepsilon(m_1)=\tau(a\cdot\tau(m)).$$
Finally, from $\eqref{eq:tauc}$ we deduce that:
\begin{equation}\label{eq:18}
\tau(a_1\cdot m)\cdot a_2\stackrel{\eqref{eq:tauc}}{=}\tau(a_1\cdot \tau(m))\cdot a_2
\end{equation}
for every $m\in M$, $a\in A$ and this, together with \eqref{eq:taug} and \eqref{eq:tauf}, implies that
\begin{displaymath}
\tau(a_1\cdot \tau(m))\cdot a_2 \stackrel{\eqref{eq:taug}}{=}\tau(a_1\cdot \tau(\tau(m)_0))\cdot a_2\tau(m)_1 \stackrel{\eqref{eq:18}}{=} \tau(a_1\cdot \tau(m)_0)\cdot a_2\tau(m)_1 \stackrel{\eqref{eq:tauf}}{=} a\cdot \tau(m).
\end{displaymath}
\end{remark}

\begin{proposition}\label{prop:preantipode}
Let $(A,m,u,\Delta,\varepsilon,\Phi)$ be a quasi-bialgebra. Then for all $M\in \quasihopf{A}$, any $\K$-linear map $\tau\colon M\to M$ satisfying \eqref{eq:taua}, \eqref{eq:tauf} and \eqref{eq:taug} of Proposition \ref{proptau} induces a map
\begin{equation}\label{eq:36}
\sfun{\tilde{\tau}}{\frac{M}{MA^+}}{M}{\cl{m}}{\tau(m)}
\end{equation}
satisfying conditions \eqref{eq:9} - \eqref{eq:10} from Theorem \ref{Th2.4}. Consequently, if such a map $\tau$ exists for all $M\in \quasihopf{A}$, then the adjunction $(F,G,\eta,\varepsilon)$ is an equivalence of categories.
\end{proposition}

\begin{proof}
Note that $\tau$ actually factors through the quotient $\frac{M}{MA^+}$, since it satisfies $\tau(m\cdot a)=\tau(m)\,\varepsilon(a)$ (cf. $\eqref{eq:taua}$). Thus the map $\tilde{\tau}$ of \eqref{eq:36} is a well-defined $\K$-linear map.
Keeping in mind Proposition \ref{proptau}, let us show that this $\tilde{\tau}$ satisfies \eqref{eq:9} and \eqref{eq:10}:
\begin{itemize}
\item For all $m\in M$
$$\tilde{\tau}(\cl{m_0})\cdot m_1=\tau(m_0)\cdot m_1\stackrel{\eqref{eq:tauf}}{=}m.$$
\item From $\eqref{eq:tauf}$ we deduce that in the quotient $\frac{M}{MA^+}$ the following relation holds
\begin{equation}\label{eq:20}
\cl{m}=\cl{\tau(m_0)\cdot m_1}=\cl{\tau(m)}
\end{equation}
for all $m\in M$. Consequently, from the identity $\tilde{\tau}(\cl{m})=\tau(m)$ together with $\eqref{eq:taug}$ and \eqref{eq:20}, we find out that:
\begin{displaymath}
\cl{\tilde{\tau}(\cl{m})_0}\otimes \tilde{\tau}(\cl{m})_1 = \cl{\tau(m)_0}\otimes \tau(m)_1 \stackrel{\eqref{eq:20}}{=} \cl{\tau(\tau(m)_0)}\otimes \tau(m)_1 \stackrel{\eqref{eq:taug}}{=} \cl{\tau(m)}\otimes 1 \stackrel{\eqref{eq:20}}{=} \cl{m}\otimes 1.
\end{displaymath}
\end{itemize}
This completes the proof.
\end{proof}

As a consequence of Proposition \ref{prop:preantipode} and Proposition \ref{proptau} we have the following central theorem.

\begin{theorem}\label{preantipode}
Let $(A,m,u,\Delta,\varepsilon,\Phi)$ be a quasi-bialgebra and $S$ be a preantipode for $A$. Then for all $M\in\quasihopf{A}$, the map
\begin{equation}\label{eq:tautilde}
\tilde{\tau}:\frac{M}{MA^+}\longrightarrow M:\cl{m}\longmapsto \Phi^1\cdot m_0\cdot S(\Phi^2 m_1)\Phi^3
\end{equation}
is $\K$-linear and satisfies conditions \eqref{eq:9} - \eqref{eq:10} from Theorem \ref{Th2.4}, so that the adjunction $(F,G,\eta,\varepsilon)$ is an equivalence of categories.
\end{theorem}

Next aim is to show that if the adjunction $(F,G,\eta,\epsilon)$ is an equivalence, then we can construct a map $S$ that satisfies the conditions \eqref{eq:P1}, \eqref{eq:P2} and \eqref{eq:P3} of Definition \ref{defpreantipode}.

Therefore, consider the quasi-Hopf bimodule $A\otimes A$ with the following structures:
$$A\hat{\otimes}A:=T({}_\circ A{}_\bullet)={}_\circ A{}_\bullet\otimes {}_\bullet^{} A{}_\bullet^\bullet$$
where the tensor product is taken in ${}_A\M{}_A$. Explicitly:
\begin{subequations}
\begin{gather}
x\cdot(a\otimes b)=a\otimes xb, \qquad (a\otimes b)\cdot x=ax_1\otimes bx_2 \label{eq:TAstructa},\\
\rho(a\otimes b)=((a\otimes b_1)\otimes b_2)\cdot \Phi=a\Phi^1\otimes b_1\Phi^2\otimes b_2\Phi^3 \label{eq:TAstructb}
\end{gather}
\end{subequations}
for all $a,b,x\in A$ (recall relations \eqref{eq:8} and \eqref{eq:Q2}). Set 
\begin{equation}\label{eq:etacap}
\sfun{\hat{\eta}_A:=\eta_{A\hat{\otimes}A}}{A\hat{\otimes }A}{\displaystyle\frac{A\hat{\otimes }A}{(A\hat{\otimes }A)A^+}\otimes A}{a\otimes b}{\cl{a\Phi^1\otimes b_1\Phi^2}\otimes b_2\Phi^3}.
\end{equation}
The structures on $\frac{A\hat{\otimes }A}{(A\hat{\otimes }A)A^+}\otimes A$ are given by:
\begin{gather*}
x\cdot(\cl{a\otimes b}\otimes c)=\cl{a\otimes x_1b}\otimes x_2c, \\
(\cl{a\otimes b}\otimes c)\cdot x=\cl{a\otimes b}\otimes cx, \\
\rho(\cl{a\otimes b}\otimes c)=(\cl{a\otimes \phi^1b}\otimes \phi^2c_1)\otimes \phi^3c_2,
\end{gather*}
for all $a,b,c,x\in A$.
If we assume that $(F,G)$ is an equivalence, then $\hat{\eta}_A$ is an isomorphism in $\quasihopf{A}$. This means that $\hat{\eta}_A^{-1}$ exists and it is an isomorphism too.
Note that since it is right $A$-linear, i.e.:
\begin{equation}\label{eq:23}
\hat{\eta}_A^{-1}(\cl{a\otimes b}\otimes c)=\hat{\eta}_A^{-1}(\cl{a\otimes b}\otimes 1)\cdot c,
\end{equation}
it is completely determined by its value on elements of the form $\cl{a\otimes b}\otimes 1$. Set
\begin{equation}\label{eq:24}
a^1\otimes a^2:=\hat{\eta}_A^{-1}(\cl{1\otimes a}\otimes 1)
\end{equation}
(summation understood) and define a new map:
\begin{equation}\label{eq:xi}
\xi:\displaystyle\frac{A\hat{\otimes }A}{(A\hat{\otimes }A)A^+}\longrightarrow A:\cl{a\otimes b}\longmapsto (A\otimes \varepsilon)\hat{\eta}_A^{-1}(\cl{a\otimes b}\otimes 1).
\end{equation}
\emph{We will see at the end of Section \ref{section5} that $\xi$ plays a central role in reconstructing the quasi-antipode from the preantipode.}

Consider the left $A$-action on $A\otimes A$ given by the multiplication on the first factor. The subset $(A\otimes A)A^+$ is still a left $A$-submodule of $A\otimes A$ and so this action passes to the quotient $\frac{A\otimes A}{(A\otimes A)A^+}$. It is not hard to see that $\hat{\eta}_A$ in left $A$-linear with respect to the multiplication on the first factor.
This implies that also $\hat{\eta}_A^{-1}$ is. In particular, for all $a,b\in A$
\begin{gather}
\hat{\eta}_A^{-1}(\cl{a\otimes b}\otimes 1)=ab^1\otimes b^2, \label{eq:26} \\
\xi(\cl{a\otimes b})=ab^1\,\varepsilon(b^2).
\end{gather}

Next observe that, in view of \eqref{eq:23}, we can write:
\begin{equation}\label{eq:etainv0}
\hat{\eta}_A^{-1}(\cl{a\otimes b}\otimes c)\stackrel{\eqref{eq:26}}{=}(ab^1\otimes b^2)\cdot c=ab^1c_1\otimes b^2c_2.
\end{equation}

Define
\begin{equation}\label{eq:defS}
S(a):=a^1\,\varepsilon(a^2)=(A\otimes \varepsilon)\left(\hat{\eta}_A^{-1}(\cl{1\otimes a}\otimes 1)\right)
\end{equation}
for all $a\in A$ and let us show that this is a preantipode for $A$.
\begin{itemize}
\item First of all, $\varfun{S}{A}{A}$ is clearly $\K$-linear. Moreover 
\begin{equation}\label{eq:xiS}
\xi(\cl{a\otimes b})=aS(b).
\end{equation}
\item For every $\cl{a\otimes b}\in\frac{A\otimes A}{(A\otimes A)A^+}$ we have that:
$$\cl{ax_1\otimes bx_2}=\cl{(a\otimes b)\cdot x}=\cl{a\otimes b}\,\varepsilon(x).$$
This implies that:
$$ax_1S(bx_2)=\xi(\cl{ax_1\otimes bx_2})=\xi(\cl{a\otimes b})\,\varepsilon(x)=aS(b)\,\varepsilon(x)$$
and, evaluating it at $a=1$, we get \eqref{eq:P1}: $x_1S(bx_2)=\varepsilon(x)S(b)$.
\item Recall that $\hat{\eta}_A^{-1}$ is left $A$-linear with respect to the original left $A$-action too. Hence:
\begin{displaymath}
\begin{split}
a^1\otimes xa^2 & =x\cdot(a^1\otimes a^2)=x\cdot \hat{\eta}_A^{-1}(\cl{1\otimes a}\otimes 1)=\hat{\eta}_A^{-1}(x\cdot(\cl{1\otimes a}\otimes 1))\\
 & =\hat{\eta}_A^{-1}(\cl{1\otimes x_1a}\otimes x_2)\stackrel{\eqref{eq:etainv0}}{=}(x_1a)^1(x_2)_1\otimes (x_1a)^2(x_2)_2.
\end{split}
\end{displaymath}
Applying $A\otimes \varepsilon$ on both sides we find \eqref{eq:P2}: $\varepsilon(x)\,S(a)=S(x_1a)x_2.$
\item Finally, $\hat{\eta}_A^{-1}$ is the inverse of $\hat{\eta}_A$. Consequently
\begin{displaymath}
\begin{split}
a\otimes b & \stackrel{\phantom{(33)}}{=} \hat{\eta}_A^{-1}(\hat{\eta}_A(a\otimes b))=\hat{\eta}_A^{-1}(\cl{a\Phi^1\otimes b_1\Phi^2}\otimes b_2\Phi^3) \\
 & \stackrel{\eqref{eq:etainv0}}{=} a\Phi^1(b_1\Phi^2)^1(b_2\Phi^3)_1\otimes (b_1\Phi^2)^2(b_2\Phi^3)_2
\end{split}
\end{displaymath}
and applying $A\otimes \varepsilon$ on both sides again we obtain:
$$a\,\varepsilon(b)=a\Phi^1S\left(b_1\Phi^2\right)b_2\Phi^3.$$
For $a=b=1$ we get \eqref{eq:P3}: $1=\Phi^1S(\Phi^2)\Phi^3.$
\end{itemize}

\begin{remark}
Observe that we can use the right $A$-colinearity of $\hat{\eta}_A^{-1}$ to express it explicitly as a function of $S$. Indeed, from the commutativity of the following diagram:
\begin{equation}\label{eq:etacolin}
\xymatrix{ \ar@{}[drr]|{\circlearrowleft}
\cl{A\hat{\otimes}A}\otimes A \ar@{|->}[rr]^{\hat{\eta}_A^{-1}} \ar@{|->}[d]_-{\rho_{\cl{A\hat{\otimes}A}\otimes A}} & & A\hat{\otimes}A \ar@{|->}[d]^-{\rho_{A\hat{\otimes}A}} \\
\cl{A\hat{\otimes}A}\otimes A \otimes A \ar@{|->}[rr]_-{\hat{\eta}_A^{-1}\otimes A} & & A\hat{\otimes}A\otimes A
}
\end{equation}
one deduces that:
\begin{displaymath}
\begin{split}
\left(\phi^1b\right)^1\left(\phi^2\right)_1\otimes\left(\phi^1b\right)^2\left(\phi^2\right)_2\otimes\phi^3 & \hspace{1.9pt}\stackrel{\eqref{eq:etainv0}}{=}\hspace{1.9pt} \hat{\eta}_A^{-1}\left(\cl{1\otimes \phi^1b}\otimes \phi^2\right)\otimes \phi^3 \\
 & \hspace{1.9pt}\stackrel{\eqref{eq:etacolin}}{=}\hspace{1.9pt} \rho_{A\hat{\otimes}A}\left(\hat{\eta}_A^{-1}(\cl{1\otimes b}\otimes 1)\right)  \\
 & \hspace{1.9pt}\stackrel{\eqref{eq:24}}{=}\hspace{1.9pt} \rho_{A\hat{\otimes}A}\left(b^1\otimes b^2\right) \\
 & \stackrel{\eqref{eq:TAstructb}}{=} \left(b^1\Phi^1\otimes (b^2)_1\Phi^2\right)\otimes (b^2)_2\Phi^3.	
\end{split}
\end{displaymath}
Applying $A\otimes \varepsilon\otimes A$ to the leftmost member and to the rightmost one and in view of relations \eqref{eq:Q2}, \eqref{eq:23} and \eqref{eq:defS} we find:
\begin{equation*}
S(\phi^1b)\phi^2\otimes \phi^3=b^1\otimes b^2=\hat{\eta}_A^{-1}(\cl{1\otimes b}\otimes 1),
\end{equation*}
so that from \eqref{eq:etainv0} we can conclude that:
\begin{equation}\label{eq:30}
\hat{\eta}_A^{-1}(\cl{a\otimes b}\otimes c)=aS(\phi^1b)\phi^2c_1\otimes \phi^3c_2.
\end{equation}
\end{remark}

\begin{theorem}\label{FundStructTheo}\emph{(Structure Theorem for quasi-Hopf bimodules)}
For a quasi-bialgebra $(A,m,u,\Delta,\varepsilon,\Phi)$ the following assertions are equivalent:
\begin{itemize}
\item[(1)] the adjunction $(F,G,\eta,\epsilon)$ is an equivalence of categories;
\item[(2)] $\hat{\eta}_A$ is bijective;
\item[(3)] there exists a preantipode;
\item[(4)] for all $M\in\quasihopf{A}$, there is a linear map $\varfun{\tau}{M}{M}$ that satisfies \eqref{eq:taua}, \eqref{eq:tauf} and \eqref{eq:taug}.
\end{itemize}
\end{theorem}

\begin{proof}
$(1)\Rightarrow (2)$. It follows from the fact that $\hat{\eta}_A=\eta_{A\hat{\otimes }A}$.\\
$(2)\Rightarrow (3)$. We proved this in the paragraph that follows Theorem \ref{preantipode}.\\
$(3)\Rightarrow (4)$. It follows from Proposition \ref{proptau}. \\
$(4)\Rightarrow (1)$. It follows from Theorem \ref{Th2.4} and Proposition \ref{prop:preantipode}.
\end{proof}

\subsection{The space of coinvariant elements of a quasi-Hopf bimodule}\label{section4.2}

A careful observer can object that, if $M$ is an Hopf module, then the ordinary Structure Theorem involves the space $\coinv{M}{A}$ and not a quotient $\frac{M}{MA^+}$. Actually, there exists a suitable extension of the notion of coinvariant elements of a quasi-Hopf bimodule such that they are isomorphic.

The results that follow have been proven for quasi-Hopf algebras (that we will introduce later) by Hausser and Nill in \cite{HN}. Here we generalize these results to the framework of quasi-bialgebras with preantipode and in the next section we will show how the original ones can be recovered from the new ones.

Let $(A,m,u,\Delta,\varepsilon,\Phi,S)$ be a quasi-bialgebra with preantipode $S$, let $M$ be a quasi-Hopf bimodule and consider the linear transformation $\varfun{\tau}{M}{M}$ of \eqref{eq:tau}. The definition below is the analogue of \cite[Definition 3.5]{HN} (see also \cite[introduction to Section 3, p. 566]{BC}).

\begin{definition}\label{def:coinv}
The \textbf{space of coinvariant elements} (or just the \emph{space of coinvariants}, for the sake of brevity) of a quasi-Hopf $A$-bimodule $M$ is defined to be $$\coinv{M}{A}:=\tau(M).$$
\end{definition}

Before showing that this definition is exactly what we need to formulate the Structure Theorem in terms of $\coinv{M}{A}$, let us retrieve a proposition that collects some characterizations of the space of coinvariants (cf. also \cite[Definition 3.5]{HN} and \cite[Corollary 3.9]{HN}). The proof is left to the reader.

\begin{proposition}
If $A$ is a quasi-bialgebra with preantipode $S$ and $M$ is a quasi-Hopf $A$-bimodule, then the following descriptions of $\coinv{M}{A}$ hold, where $\phi^1\otimes \phi^2\otimes\phi^3=\Phi^{-1}$:
\begin{equation}\label{eq:coinvariants}
\begin{split}
\coinv{M}{A} & =\left\{n\in M\mid \tau(n)=n\right\} \\
 & =\left\{n\in M\mid \tau(n_0)\otimes n_1=\tau(n)\otimes 1\right\} \\
 & =\left\{n\in M\mid \rho_M(n)=\tau(\phi^1\cdot n)\cdot \phi^2\otimes \phi^3\right\}.
\end{split}
\end{equation}

\end{proposition}

\begin{remark}
Note that if $n\in M$ is such that $\rho_M(n)=n\otimes 1$, then $n\in\coinv{M}{A}$. However, up to this moment, we found no evidence that the converse is true or not, in general.
\end{remark}

Next lemma is the analogue of \cite[Lemma 3.6]{HN} for a quasi-bialgebra $A$ with preantipode $S$.

\begin{lemma}
Let $M$ be a left $A$-module. Then the coinvariants of the quasi-Hopf bimodule $G(M)={}_\bullet^{}M{}_\circ^{}\otimes {}_\bullet^{}A{}_\bullet^\bullet$ are given by $\coinv{(M\otimes A)}{A}=M\otimes \K,$ and for $m\in M$ and $a\in A$ we have that $\tau(m\otimes a)=m\otimes \varepsilon(a)$.
\end{lemma}

\begin{proof}
In view of $\eqref{eq:taua}$ of Proposition \ref{proptau}, we have that
\begin{equation}\label{eq:3.60}
\tau(m\otimes a)=\tau((m\otimes 1)\cdot a)=\tau(m\otimes 1)\,\varepsilon(a).
\end{equation}
Moreover, denoting by $\phi^1\otimes \phi^2\otimes \phi^3=\Phi^{-1}$ the inverse of $\Phi=\Phi^1\otimes \Phi^2\otimes \Phi^3$, we get:
\begin{equation}\label{eq:3.61}
\begin{split}
\tau(m\otimes 1) & \stackrel{\phantom{(19)}}{=}\Phi^1\cdot(m\otimes 1)_0\cdot S(\Phi^2(m\otimes 1)_1)\Phi^3 \stackrel{\eqref{eq:7c}}{=} \Phi^1\cdot (\phi^1m\otimes \phi^2)\cdot S(\Phi^2\phi^3)\Phi^3 \\
 & \hspace{1.6pt}\stackrel{(*)}{=} \phi^1\Phi^1\cdot m\otimes\phi^2\Psi^1(\Phi^2)_1S\left((\phi^3)_1\Psi^2(\Phi^2)_2\right)(\phi^3)_2\Psi^3\Phi^3 \\
 & \stackrel{\eqref{eq:P2}}{=} \phi^1\Phi^1\cdot m\otimes\phi^2\Psi^1\,\varepsilon(\Phi^2)\,S\left((\phi^3)_1\Psi^2\right)(\phi^3)_2\Psi^3\Phi^3 \\
 & \stackrel{\eqref{eq:P1}}{=} \phi^1\Phi^1\cdot m\otimes\phi^2\Psi^1\,\varepsilon(\Phi^2)\,S\left(\Psi^2\right)\,\varepsilon(\phi^3)\,\Psi^3\Phi^3 \\
 & \hspace{1.6pt}\stackrel{\eqref{eq:Q2}}{=} m\otimes \Psi^1S(\Psi^2)\Psi^3 \stackrel{\eqref{eq:P3}}{=} m\otimes 1
\end{split}
\end{equation}
where, again, we denoted by $\Psi$ another copy of $\Phi$ in order to avoid confusion and in $(*)$ we used \eqref{eq:Q1} in the form: $$(\Delta\otimes A\otimes A)(\Phi)\cdot(\Phi^{-1}\otimes 1)=(A\otimes A\otimes \Delta)(\Phi^{-1})\cdot(1\otimes \Psi)\cdot(A\otimes \Delta\otimes A)(\Phi).$$
Relation \eqref{eq:3.61} implies that $M\otimes \K\subseteq \coinv{(M\otimes A)}{A}$. Furthermore, combined with \eqref{eq:3.60}, it shows that if $m\otimes a\in\coinv{(M\otimes A)}{A}$, then $$m\otimes a=\tau(m\otimes a)=m\otimes \varepsilon(a)=m\,\varepsilon(a)\otimes 1.$$
Hence $M\otimes \K\supseteq \coinv{(M\otimes A)}{A}$.
\end{proof}

Now we are ready to prove that the quotient $\overline{M}$ that occurs in the Structure Theorem \ref{FundStructTheo} is isomorphic (as left $A$-module) to $\coinv{M}{A}$.

\begin{proposition}\label{prop3.2.31}
Let $(A,m,u,\Delta,\varepsilon,\Phi,S)$ be a quasi-bialgebra with preantipode $S$ and consider $\coinv{M}{A}$ as a left $A$-module with $A$-action given by \eqref{eq:17}. Then the linear map $\tilde{\tau}$, as defined in \eqref{eq:tautilde}, induces a map
$$\varfun{\tilde{\tau}}{\frac{M}{MA^+}}{\coinv{M}{A}}$$
that we denote by $\tilde{\tau}$ as well. This $\tilde{\tau}$ comes out to be an isomorphism of left $A$-modules with inverse given by $$\sfun{\sigma}{\coinv{M}{A}}{\displaystyle\frac{M}{MA^+}}{m}{\cl{m}}.$$
\end{proposition}

\begin{proof}
Let us begin by showing that the induced map $\tilde{\tau}$ is bijective. Indeed, for all $m\in M$ and for all $n\in\coinv{M}{A}$
\begin{displaymath}
\sigma\left(\tilde{\tau}\left(\cl{m}\right)\right)=\cl{\tilde{\tau}(\cl{m})}=\cl{\tau(m)}\stackrel{\eqref{eq:20}}{=}\cl{m} \qquad\mathrm{and}\qquad \tilde{\tau}(\sigma(n))=\tilde{\tau}(\cl{n})=\tau(n)\stackrel{\eqref{eq:coinvariants}}{=}n.
\end{displaymath}
Next, consider the left $A$-action on $\coinv{M}{A}$ given by \eqref{eq:17}: $a\blacktriangleright m:=\tau(a\cdot m)$
for all $a\in A$ and $m\in \coinv{M}{A}$. Actually, it is an action. Indeed, by the definition of $\coinv{M}{A}=\tau(M)$ and in view of $\eqref{eq:taud}$ of Proposition \ref{proptau}, in order to prove it it's enough to verify that $1\blacktriangleright m=m,$
but
\begin{displaymath}
1\blacktriangleright m=\tau(1\cdot m)=\tau(m)\stackrel{\eqref{eq:coinvariants}}{=}m
\end{displaymath}
for all $m\in\coinv{M}{A}$. Moreover, $\eqref{eq:tauc}$ of Proposition \ref{proptau} guarantees that $\tau$ is $A$-linear with respect to this left $A$-action and hence $\tilde{\tau}$, too.
\end{proof}

As a corollary, we get the following analogue of \cite[Theorem 3.8]{HN}.

\begin{corollary}\label{th3.8HN}
Let $(A,m,u,\Delta,\varepsilon,\Phi,S)$ be a quasi-bialgebra with preantipode. Let $M$ be a quasi-Hopf $A$-bimodule. Consider $N:=\coinv{M}{A}$ as a left $A$-module with $A$-action $\blacktriangleright$ as in \eqref{eq:17}, and ${}_\bullet^{} N{}_\circ^{}\otimes {}_\bullet^{}A_\bullet^\bullet$ as a quasi-Hopf $A$-bimodule with structures indicated by the dots. Then:
$$\fun{\nu}{N\otimes A}{M}{n\otimes a}{n\cdot a}$$
provides an isomorphism of quasi-Hopf $A$-bimodules with inverse given by
$\nu^{-1}(m)=\tau(m_0)\otimes m_1$.
\end{corollary}

\begin{proof}
In view of Proposition \ref{prop3.2.31},
$\eta_M^{-1}\circ (\sigma\otimes A)$ is an isomorphism and
$$\left(\eta_M^{-1}\circ (\sigma\otimes A)\right)(n\otimes a)\stackrel{\eqref{eq:etainv}}{=}
n\cdot a$$
for all $n\in \coinv{M}{A}$ and $a\in A$. The inverse is given for all $m\in M$ by $$((\tilde{\tau}\otimes A)\circ \eta_M)(m)=\tilde{\tau}(\cl{m_0})\otimes m_1=\tau(m_0)\otimes m_1.$$
\end{proof}

\subsection{Uniqueness and gauge transformation}\label{section4.3}

Let us conclude this section with two important results about the preantipode. First of all its uniqueness and, secondly, the fact that quasi-bialgebras with preantipode form a class of quasi-bialgebras closed under gauge transformation.

\begin{theorem}\label{th3.3.9}
Let $(A,m,u,\Delta,\varepsilon,\Phi)$ be a quasi-bialgebra. If there exists a preantipode $S$ for $A$, then it is unique.
\end{theorem}

\begin{proof}
Assume that $S$ and $T$ are both preantipodes for $A$. Then we know that the adjunction $(F,G,\eta,\epsilon)$ defined in Theorem \ref{thAdj} is an equivalence of categories and that the unit $\eta$ is a natural isomorphism. Furthermore, in view of \eqref{eq:30}, we can express explicitly $\hat{\eta}^{-1}_A$ as a function of $S$ on the one side:
\begin{equation*}
\hat{\eta}_A^{-1}(\cl{a\otimes b}\otimes c)=aS(\phi^1b)\phi^2c_1\otimes \phi^3c_2
\end{equation*}
for all $a,b,c\in A$, and as a function of $T$ on the other side:
\begin{equation*}
\hat{\eta}_A^{-1}(\cl{a\otimes b}\otimes c)=aT(\phi^1b)\phi^2c_1\otimes \phi^3c_2
\end{equation*}
for all $a,b,c\in A$. In particular, by uniqueness of the inverse, for $a=c=1$ we have that:
\begin{equation*}
S(\phi^1b)\phi^2\otimes \phi^3=T(\phi^1b)\phi^2\otimes \phi^3
\end{equation*}
for all $b\in A$. If we apply $A\otimes\varepsilon$ on both sides of this relation and we simplify in view of
\begin{equation*}
(A\otimes A\otimes\varepsilon)(\Phi^{-1})\stackrel{\eqref{eq:Q2}}{=}1\otimes 1,
\end{equation*}
then we find that $S(b)=T(b)$ for all $b\in A$, as we claimed.
\end{proof}

\begin{proposition}\label{prop3.3.16}
Let $(A,m,u,\Delta,\varepsilon,\Phi,S)$ be a quasi-bialgebra with preantipode and $F\in A\otimes A$ be a gauge transformation on $A$. Define, for $a\in A$,
\begin{equation}\label{eq:gauge}
S_F(a):=F^1S(f^1aF^2)f^2.
\end{equation}
Then $(A_F,m,u,\Delta_F,\varepsilon,\Phi_F,S_F)$ is a quasi-bialgebra with preantipode.
\end{proposition}

\begin{proof}
One easily checks that $S_F$ satisfies \eqref{eq:P1}, \eqref{eq:P2} and \eqref{eq:P3} of Definition \ref{defpreantipode}.
\end{proof}

\section{Quasi-Hopf algebras and some (new) classical results}\label{section5}

Our aim here is to prove that the Structure Theorem is in accordance with the classical results. This is why we open a digression on quasi-Hopf algebras.

\begin{definition}(\cite[p. 1424]{Dri}) A quasi-bialgebra $(A,m,u,\Delta,\varepsilon,\Phi)$ is a \textbf{quasi-Hopf algebra} if there exist elements $\alpha$ and $\beta$ in $A$ and an antimultiplicative endomorphism $s$ of $A$ such that:
\begin{gather}
s(a_1)\alpha a_2=\varepsilon(a)\,\alpha, \label{eq:QHA1} \\
a_1\beta s(a_2)=\varepsilon(a)\,\beta, \label{eq:QHA2} \\
\Phi^1\beta s(\Phi^2)\alpha\Phi^3=1, \label{eq:QHA3} \\
s(\phi^1)\alpha\phi^2\beta s(\phi^3)=1, \label{eq:QHA4}
\end{gather}
where, as usual, $\Phi=\Phi^1\otimes \Phi^2\otimes \Phi^3$ and $\Phi^{-1}=\phi^1\otimes \phi^2\otimes \phi^3$. The triple $(s,\alpha,\beta)$ is usually called \textbf{antipode} \cite{HN} or \textbf{quasi-antipode} \cite{Sch1}. We will use the second terminology, in order to distinguish this one from the ordinary antipode of a Hopf algebra.
\end{definition}

\begin{remark}(\cite[Proposition 1.1]{Dri})
A quasi-antipode for a quasi-Hopf algebra is not unique, but just uniquely determined up to an invertible element. This means that if $(s,\alpha,\beta)$ and $(s',\alpha',\beta')$ are quasi-antipodes for $(A,m,u,\Delta,\varepsilon,\Phi)$ then there exists an invertible element $u\in A$ such that for all $h\in A$
\begin{equation}\label{eq:5.55}
s'(h)=us(h)u^{-1}, \quad \alpha'=u\alpha \quad\mathrm{and}\quad \beta'=\beta u^{-1}.
\end{equation}
\end{remark}

\begin{theorem}\label{th3.2.22}
Let $(A,m,u,\Delta,\varepsilon,\Phi,s,\alpha,\beta)$ be a quasi-Hopf algebra. The map $\varfun{S}{A}{A}$ defined by $S(a)=\beta s(a)\alpha$ for all $a\in A$ is a preantipode for $A$.
\end{theorem}

\begin{proof} We just need to check that the axioms are satisfied, whence we verify \eqref{eq:P1}, \eqref{eq:P2} and \eqref{eq:P3} in the given order:
\begin{gather*}
b_1S(ab_2)=b_1\beta s(ab_2)\alpha=b_1\beta s(b_2)s(a)\alpha\stackrel{\eqref{eq:QHA2}}{=}\varepsilon(b)\beta s(a)\alpha=\varepsilon(b)S(a),\\
S(a_1b)a_2=\beta s(b)s(a_1)\alpha a_2\stackrel{\eqref{eq:QHA1}}{=}\beta s(b)\alpha \varepsilon(a)=S(b)\varepsilon(a), \\
\Phi^1S(\Phi^2)\Phi^3=\Phi^1\beta s(\Phi^2)\alpha\Phi^3\stackrel{\eqref{eq:QHA3}}{=}1.
\end{gather*}
\end{proof}

\begin{corollary}\label{cor3.2.27}
Let $(A,m,u,\Delta,\varepsilon,\Phi,s,\alpha,\beta)$ be a quasi-Hopf algebra. Then the adjunction $(F,G,\eta,\epsilon)$ of Theorem \ref{thAdj} is an equivalence of categories.
\end{corollary}

By the foregoing we have that every quasi-Hopf algebra is a quasi-bialgebra with preantipode. It is more than likely that the converse does not hold, as we conjecture here, even if we are not able to provide an example at the moment. 

\begin{conjecture}
There is a quasi-bialgebra with preantipode which is not a quasi-Hopf algebra.
\end{conjecture}

Our conjecture is supported by the fact that there exists an example of a dual quasi-bialgebra with preantipode that is not a dual quasi-Hopf algebra. The interested reader may refer to \cite[Example 4.5.1]{Sch3}, where Schauenburg exhibits a dual quasi-bialgebra $H$ that does not admit a quasi-antipode but such that the category ${}^H_{}\M{}_f^{}$ of finite-dimensional left $H$-comodules is left and right rigid. By the left-handed version of \cite[Theorem 2.6]{Sch1}, this is equivalent to say that the adjunction $(F,G)$ of \cite[Theorem 2.7]{AP1} is an equivalence of categories and hence, by \cite[Theorem 3.9]{AP1}, $H$ admits a preantipode (cf. also \cite[Remark 3.12]{AP1} and \cite[Remark 2.17]{AP2}).

\begin{remark}
An example of a quasi-bialgebra $A$ with preantipode but without quasi-antipode has to be a noncommutative quasi-bialgebra (cf. Corollary \ref{cor3.3.11}) with infinite-dimensional underlying vector space (cf. Theorem \ref{thSch}) such that there is no gauge transformation $F$ for which $A$ is the twist of a quasi-Hopf algebra via $F$ (cf. \cite[Remark 5, p. 1425]{Dri}).
\end{remark}

\begin{remark}\label{rem3.3.13}
Let $(H,m,u,\Delta,\varepsilon,s)$ be an ordinary Hopf algebra. Set $\Phi=1\otimes 1\otimes 1$ and $\alpha=\beta=1$. Thus $(H,m,u,\Delta,\varepsilon,\Phi,s,\alpha,\beta)$ is a quasi-Hopf algebra with quasi-antipode $(s,1,1)$. By Theorem \ref{th3.2.22}, $s$ is a preantipode. Furthermore, observe that if $n\in M$ is such that $\tau(n)=n$, then
$$\rho_M(n)=n_0\otimes n_1\stackrel{\eqref{eq:coinvariants}}{=}\tau(\phi^1\cdot n)\cdot \phi^2\otimes \phi^3=\tau(n)\otimes 1=n\otimes 1,$$
so that the ordinary definition of coinvariants and Definition \ref{def:coinv} coincides. As a consequence, we can apply Corollary \ref{th3.8HN} to recover the ordinary Structure Theorem for Hopf modules (\cite[Theorem 4.1.1]{Swe}). Indeed, for every Hopf module $M$, $M\cong\coinv{M}{H}\otimes H$ via the isomorphisms:
\begin{gather*}
\lfun{\nu}{\coinv{M}{H}\otimes H}{M}{m\otimes h}{m\cdot h}, \\
\lfun{\nu^{-1}}{M}{\coinv{M}{H}\otimes H}{m}{\tau(m_0)\otimes m_1}.
\end{gather*}
Moreover, $\tau(m)=\Phi^1\cdot m_0\cdot s(\Phi^2m_1)\Phi^3=m_0\cdot s(m_1)$ implies that $\nu^{-1}(m)=m_0\cdot s(m_1)\otimes m_2$.
\end{remark}

\begin{remark}\label{rem3.3.14}
Let $(A,m,u,\Delta,\varepsilon,\Phi,s,\alpha,\beta)$ be a quasi-Hopf algebra with quasi-antipode $(s,\alpha,\beta)$ and assume that $s$ is bijective. By Theorem \ref{th3.2.22}, $S(\cdot)=\beta s(\cdot)\alpha$ is a preantipode. Thus the map $\tau$ has the form:
$$\tau(m)=\Phi^1\cdot m_0\cdot \beta s(\Phi^2m_1)\alpha \Phi^3=\Phi^1\cdot m_0\cdot \beta s(s^{-1}(\alpha\Phi^3)\Phi^2m_1).$$
This $\tau$ is exactly the projection $E$ of Hausser and Nill and $\coinv{M}{A}$, obtained as image of $\tau$, is the same $\coinv{M}{A}$ that appears in \cite[Definition 3.5]{HN} and \cite[Corollary 3.9]{HN}. Moreover, the $A$-action $\blacktriangleright$ coincides with the action they indicate with $\triangleright$ and Corollary \ref{th3.8HN} is precisely \cite[Theorem 3.8]{HN}.
\end{remark}

This last remark and the previous one show how the theory we developed here latch on to the traditional results about Hopf and quasi-Hopf bimodules.

\begin{remark}\label{rem3.3.15}
In \cite{BC} a different space of coinvariant elements is introduced for a quasi-Hopf $A$-bimodule over a quasi-Hopf algebra $(A,m,u,\Delta,\varepsilon,\Phi,s,\alpha,\beta)$. 

Define $T\otimes U\otimes V\otimes W:=(1\otimes \Phi^{-1})\cdot (A\otimes A\otimes \Delta)(\Phi)$ (summation understood) and the right $A\otimes A\otimes A\otimes A$-action on $A\otimes A$ given on generators by:
$$(a\otimes b)\triangleleft(c\otimes d\otimes e\otimes f):=s(d)ae\otimes s(c)bf$$
and extended by linearity (cf. \cite[proof of Lemma 1, p. 1427]{Dri}). Next, set
$$\gamma=\gamma^1\otimes \gamma^2:=(\alpha\otimes \alpha)\triangleleft(T\otimes U\otimes V\otimes W)=s(U)\alpha V\otimes s(T)\alpha W$$
and
\begin{equation*}
\begin{split}
f & =f^1\otimes f^2 :=(s\otimes s)(\Delta^{op}(\phi^1))\cdot\gamma\cdot \Delta(\phi^2\beta s(\phi^3)) \\
 & = (\alpha\otimes \alpha)\triangleleft((1\otimes \Phi^{-1})\cdot (A\otimes A\otimes \Delta)(\Phi)\cdot (\Delta\otimes \Delta)(p_R)),
\end{split}
\end{equation*}
where $p_R=p^1\otimes p^2:=\phi^1\otimes \phi^2\beta s(\phi^3)$.

Then this other space of coinvariants is defined to be the set:
$$M^{\overline{coA}}:=\left\{n\in M\mid \rho(n)=\phi^1\cdot n\cdot s(\phi^3_2\Phi^3)f^1\otimes \phi^2\Phi^1\beta s(\phi^3_1\Phi^2)f^2\right\}.$$
Furthermore, a new map
$$\lfun{\overline{E}}{M}{M}{m}{m_0\cdot \beta s(m_1)}$$
is defined, that should be the analogue of the projection $E$ of \cite{HN} and it is connected with $\tau$ via:
$$\overline{E}(m)=\tau(p^1\cdot m)\cdot p^2 \qquad\mathrm{and}\qquad \tau(m)=\Phi^1\cdot \overline{E}(m)\cdot s(\Phi^2)\alpha\Phi^3$$
for all $m\in M$. Bulacu and Caenepeel proved that (cf. \cite[Lemma 3.6]{BC})
$$M^{\overline{\mathrm{co}A}}=\left\{n\in M:\overline{E}(n)=n\right\},$$
that $M^{\overline{\mathrm{co}A}}$ is an $A$-submodule of $M$ with respect to the left adjoint $A$-action $h\triangleright m:=h_1\cdot m\cdot s(h_2)$ and that $\coinv{M}{A}$ and $M^{\overline{\mathrm{co}A}}$ are isomorphic as left $A$-modules via $\tau$ (caveat: not $\tilde{\tau}$) and $\overline{E}$, i.e.
\begin{equation*}
\varfun{\overline{E}}{\coinv{M}{A}}{M^{\overline{\mathrm{co}A}}}\qquad\mathrm{and}\qquad \varfun{\tau}{M^{\overline{\mathrm{co}A}}}{\coinv{M}{A}}
\end{equation*}
are inverses of each others.

Through the definition of this $M^{\overline{\mathrm{co}A}}$, they were able to prove another structure theorem for quasi-Hopf bimodules. Nominally (cf. \cite[Theorem 3.7]{BC}) every quasi-Hopf $A$-bimodule $M$ is isomorphic to $M^{\overline{\mathrm{co}A}}\otimes A$ via
\begin{gather*}
\sfun{\overline{\nu}}{M^{\overline{\mathrm{co}A}}\otimes A}{M}{n\otimes h}{\Phi^1\cdot n\cdot s(\Phi^2)\alpha\Phi^3h}, \\
\sfun{\overline{\nu}^{-1}}{M}{M^{\overline{\mathrm{co}A}}\otimes A}{m}{\overline{E}(m_0)\otimes m_1}.
\end{gather*}
The proof they gave relies broadly on the Structure Theorem of Hausser and Nill (\cite[Theorem 3.8]{HN}). Nevertheless, they claimed in \cite[Remark 3.8]{BC} that their result has a direct proof that doesn't involve the bijectivity of the quasi-antipode and Hausser-Nill Structure Theorem, but also that that proof is long and technical. Here we found a bijectivity-free proof	of their theorem, that has the advantage of not being long nor technical. Indeed, in Remark \ref{rem3.3.14} the bijectivity of $s$ is definitely unnecessary and we have the following commutative diagram:
\begin{displaymath}
\xymatrix @C=40pt @R=20pt {
\overline{M}\otimes A \ar[r]^-{\eta_M^{-1}} \ar[d]_-{\tilde{\tau}\otimes A} & M \\
\coinv{M}{A}\otimes A \ar[r]_-{\overline{E}\otimes A} & M^{\overline{\mathrm{co}A}}\otimes A \ar[u]_-{\overline{\nu}}
}
\end{displaymath}
for any quasi-Hopf bimodule $M$. Unfortunately, it is not clear to us if $M^{\overline{\mathrm{co}A}}$ has an analogue for quasi-bialgebras with preantipode.
\end{remark}

The following lemma comes from \cite[Example 2.4.1]{Maj} and it is retrieved here because it allows us to show that a preantipode is neither antimultiplicative nor anticomultiplicative in general.

\begin{lemma}\label{lemmaMajid}
Let $(H,m,u,\Delta,\varepsilon,s)$ be an ordinary Hopf algebra. Let $\Phi\in H\otimes H\otimes H$ be an invertible element that satisfies \eqref{eq:Q1}, \eqref{eq:Q2} and \eqref{eq:Q3}.
Assume that $c:=\Phi^1s(\Phi^2)\Phi^3\in H$ is invertible with inverse $\beta$ and set $\alpha=1$. Then $(H,m,u,\Delta,\varepsilon,\Phi,s,\alpha,\beta)$ is a quasi-Hopf algebra and $\beta\in \mathcal{Z}(H)$ where $\mathcal{Z}(H)$ is the center of $H$. Furthermore, $(H,m,u,\Delta,\varepsilon,\Phi,S)$ is a quasi-bialgebra with preantipode defined by $S(h)=\beta s(h)$, for each $h\in H$.
\end{lemma}

\begin{proof}
Obviously, if $\Phi$ satisfies \eqref{eq:Q1}, \eqref{eq:Q2} and \eqref{eq:Q3}, thus $(H,m,u,\Delta,\varepsilon,\Phi)$ is a quasi-bialgebra. Then let us show that $c\in Z(H)$. Apply $m\circ(m\otimes H)\circ(H\otimes s\otimes H)$ on both sides of \eqref{eq:Q3} to obtain that
\begin{equation}\label{eq:48}
h_1\Phi^1s(\Phi^2)s(h_2)h_3\Phi^3=\Phi^1h_1s(h_2)s(\Phi^2)\Phi^3h_3.
\end{equation}
Since we know that
\begin{equation}\label{eq:antipode}
s(h_1)h_2=\varepsilon(h)1_H=h_1s(h_2)
\end{equation}
for all $h\in H$ by definition of antipode, we can simplify \eqref{eq:48} to conclude that $hc=ch$ for all $h\in H$. Thus $\beta \in \mathcal{Z}(H)$. Moreover, consider \eqref{eq:Q1} in the form:
$$(H\otimes H\otimes \Delta)(\Phi)\cdot(\Delta\otimes H\otimes H)(\Psi)\cdot(\Phi^{-1}\otimes 1)=(1\otimes \Phi)\cdot(H\otimes \Delta\otimes H)(\Psi)$$
where we set $\Psi=\Psi^1\otimes \Psi^2\otimes \Psi^3=\Phi$, and apply $m\circ (m\otimes m)\circ(s\otimes H\otimes s\otimes H)$ on both sides:
\begin{multline*}
s(\phi^1)s((\Psi^1)_1)s(\Phi^1)\Phi^2(\Psi^1)_2\phi^2s(\phi^3)s(\Psi^2)s((\Phi^3)_1)(\Phi^3)_2\Psi^3 \\
= s(\Psi^1)\Phi^1(\Psi^2)_1s((\Psi^2)_2)s(\Phi^2)\Phi^3\Psi^3.
\end{multline*}
Simplifying once again in view of \eqref{eq:antipode} and \eqref{eq:Q2} we find that
\begin{equation}\label{eq:3.65}
s(\phi^1)\phi^2 s(\phi^3)=s(\phi^1)s((\Psi^1)_1)(\Psi^1)_2\phi^2s(\phi^3)s(\Psi^2)\Psi^3\stackrel{\eqref{eq:antipode}}{=}\Phi^1s(\Phi^2)\Phi^3=c.
\end{equation}

These computations allow us to verify easily that the axioms of a quasi-Hopf algebra are satisfied:
\begin{gather*}
s(a_1)\alpha a_2=s(a_1)a_2=\varepsilon(a)1=\varepsilon(a)\alpha, \\
a_1\beta s(a_2)=a_1s(a_2)\beta=\varepsilon(a)\beta, \\
\Phi^1\beta s(\Phi^2)\alpha\Phi^3=\beta\Phi^1s(\Phi^2)\Phi^3=\beta c=1, \\
s(\phi^1)\alpha\phi^2\beta s(\phi^3)=\beta s(\phi^1)\phi^2s(\phi^3)\stackrel{\eqref{eq:3.65}}{=}\beta c=1.
\end{gather*}
\end{proof}

\begin{remark}\label{rem3.3.19}
As we said before, the previous example allows us to observe that $S$ is not, in general, an antiendomorphism of algebras, since
$$S(ab)=\beta s(ab)\stackrel{(*)}{=}\beta s(b)s(a)=\beta s(b)c\beta s(a)=cS(b)S(a),$$
nor an antiendomorphism of coalgebras, since
$$\Delta(S(a))=\Delta(\beta)\cdot \Delta(s(a))\stackrel{(*)}{=}\Delta(\beta)\cdot(s(a_2)\otimes s(a_1))=\Delta(\beta)\cdot(c\otimes c)\cdot (S(a_2)\otimes S(a_1)),$$
where in $(*)$ we used the fact that the antipode is an antiendomorphism of bialgebras.
Actually, in this particular situation, it depends on $\beta$ and $c$.
\end{remark}

\subsection{The other way around: from preantipodes to quasi-antipodes}\label{section5.1}

Even though we claimed that a quasi-bialgebra with preantipode is not, in general, a quasi-Hopf algebra, there exist partial converses to Theorem \ref{th3.2.22}. The subsequent proposition will retrieve an easy one, but before we need a technical lemma. Moreover, we are going to conclude this last section with some considerations concerning a result, due to Schauenburg, that proves that in the finite-dimensional case these two concepts are equivalent, and with some examples in which this equivalence is explicit.

\begin{lemma}
Let $(A,m,u,\Delta,\varepsilon,\Phi,S)$ be a quasi-bialgebra with preantipode $S$. Then:
\begin{equation} \label{eq:3.39}
S(\phi^1)\phi^2S(\phi^3)=S(1).
\end{equation}
\end{lemma}

\begin{proof}
\emph{In what follows we are going to indicate with $\Psi^{-1}=\psi^1\otimes \psi^2\otimes \psi^3$ another copy of $\Phi^{-1}$.} In view of relation \eqref{eq:Q1}, we have the following identity:
$$(\Delta\otimes A\otimes A)(\Phi^{-1})\cdot(A\otimes A\otimes \Delta)(\Psi^{-1})\cdot(1\otimes \Phi)=(\Phi^{-1}\otimes 1)\cdot(A\otimes \Delta \otimes A)(\Psi^{-1}),$$
i.e.
$$(\phi^1)_1\psi^1\otimes (\phi^1)_2\psi^2\Phi^1\otimes \phi^2(\psi^3)_1\Phi^2\otimes \phi^3(\psi^3)_2\Phi^3=\phi^1\psi^1\otimes \phi^2(\psi^2)_1\otimes \phi^3(\psi^2)_2\otimes \psi^3.$$
If we apply $S\otimes A\otimes S\otimes A$ on both sides and then we multiply we get that
\begin{displaymath}
S((\phi^1)_1\psi^1)(\phi^1)_2\psi^2\Phi^1 S(\phi^2(\psi^3)_1\Phi^2)\phi^3(\psi^3)_2\Phi^3=S(\phi^1\psi^1)\phi^2(\psi^2)_1S(\phi^3(\psi^2)_2)\psi^3
\end{displaymath}
and we can simplify it in view of \eqref{eq:P1}, \eqref{eq:P2} and \eqref{eq:Q2} to obtain:
\begin{displaymath}
S(\psi^1)\psi^2\Phi^1S((\psi^3)_1\Phi^2)(\psi^3)_2\Phi^3=S(\phi^1)\phi^2S(\phi^3).
\end{displaymath}
Therefore, simplifying this further in view of \eqref{eq:P2}, \eqref{eq:P3} and \eqref{eq:Q2} again, we can conclude that $S(1)=S(\phi^1)\phi^2S(\phi^3)$, as we claimed.
\end{proof}

\begin{proposition}\label{prop3.2.30}
Let $(A,m,u,\Delta,\varepsilon,\Phi,S)$ be a quasi-bialgebra with preantipode. If $\Phi$ is in the center of $A\otimes A\otimes A$, then $(A,m,u,\Delta,\varepsilon,s)$ is an ordinary Hopf algebra where
\begin{equation}\label{eq:s}
s(a)=\Phi^1S(a\Phi^2)\Phi^3,
\end{equation}
for all $a\in A$. Furthermore $(A,m,u,\Delta,\varepsilon,\Phi,s,\alpha,\beta)$ is a quasi-Hopf algebra with $\alpha=1$ and $\beta=S(1)$. Moreover, for all $a\in A$ one has
\begin{equation}\label{eq:3.51}
S(a)=\beta s(a).
\end{equation}
\end{proposition}

\begin{proof}
In view of \eqref{eq:Q4} and \eqref{eq:Q5}, we know that $\varepsilon$ is a counit for $\Delta$. Moreover, commutativity of $\Phi$ ensures that $\Delta$ is coassociative. Indeed, by \eqref{eq:Q3}:
$$(\Delta\otimes A)(\Delta(a))=\Phi^{-1}((A\otimes \Delta)(\Delta(a)))\Phi=((A\otimes \Delta)(\Delta(a)))\Phi^{-1}\Phi=(A\otimes \Delta)(\Delta(a))$$
for every $a\in A$, so that $(A,m,u,\Delta,\varepsilon)$ is an ordinary bialgebra. Let us show that $s$ is an antipode:
\begin{displaymath}
\begin{split}
(s*\id)(a) & \stackrel{\phantom{(17)}}{=} s(a_1)a_2=\Phi^1S(a_1\Phi^2)\Phi^3a_2\stackrel{(*)}{=}\Phi^1S(a_1\Phi^2)a_2\Phi^3 \\
 & \stackrel{\eqref{eq:P2}}{=} \Phi^1S(\Phi^2)\Phi^3\,\varepsilon(a) \stackrel{\eqref{eq:P3}}{=} (u\circ\varepsilon)(a)
\end{split}
\end{displaymath}
where in $(*)$ we used that $\Phi$ is in the center in the following form:
\begin{equation}\label{eq:eq3.59}
(\Phi^1\otimes a_1\otimes \Phi^2\otimes \Phi^3a_2)=(\Phi^1\otimes a_1\otimes \Phi^2\otimes a_2\Phi^3).
\end{equation}
Analogously:
\begin{displaymath}
\begin{split}
(\id*s)(a) & \stackrel{\phantom{(17)}}{=} a_1s(a_2)=a_1\Phi^1S(a_2\Phi^2)\Phi^3\stackrel{(**)}{=}\Phi^1a_1S(\Phi^2a_2)\Phi^3 \\
 & \stackrel{\eqref{eq:P1}}{=} \Phi^1S(\Phi^2)\Phi^3\,\varepsilon(a) \stackrel{\eqref{eq:P3}}{=} (u\circ\varepsilon)(a)
\end{split}
\end{displaymath}
where $(**)$ follows from $(\Delta(a)\otimes 1)\Phi=\Phi(\Delta(a)\otimes 1).$
Hence $(A,m,u,\Delta,\varepsilon,s)$ is an ordinary Hopf algebra. Moreover:
\begin{displaymath}
S(a)=S(a_1\varepsilon(a_2))=S(a_1)\varepsilon(a_2)\stackrel{(\circ)}{=}S(a_1)a_2s(a_3)\stackrel{\eqref{eq:P2}}{=}S(1)\varepsilon(a_1)s(a_2)=\beta s(a)
\end{displaymath}
where in $(\circ)$ we used that $\id*s=u\circ\varepsilon$ and coassociativity of $\Delta$ to renumber. Now, let us show that $(A,m,u,\Delta,\varepsilon,\Phi,s,\alpha,\beta)$ is a quasi-Hopf algebra:
\begin{itemize}
\item We know that $s$ is an antiendomorphism of $A$, since it is an ordinary antipode.
\item Since $\alpha=1$, we have $s(a_1)\alpha a_2=s(a_1)a_2=\varepsilon(a)1=\varepsilon(a)\alpha$.
\item In view of \eqref{eq:3.51} we get that $a_1\beta s(a_2)=a_1S(a_2)\stackrel{\eqref{eq:P2}}{=}\varepsilon(a)S(1)=\varepsilon(a)\beta.$
\item By \eqref{eq:3.51} again we have that $\Phi^1\beta s(\Phi^2)\alpha \Phi^3=\Phi^1 S(\Phi^2)\Phi^3=1$.
\item In order to prove \eqref{eq:QHA4}, first apply $m\circ (m\otimes A)\circ(S\otimes A\otimes S)$ on both sides of
$$(\Delta\otimes A)(\Delta(a))\Phi^{-1}\stackrel{\eqref{eq:Q3}}{=}\Phi^{-1}(A\otimes \Delta)(\Delta(a))$$
and then simplify in view of \eqref{eq:P1} and \eqref{eq:P2} to get that
\begin{equation}\label{eq:eq3.60}
S(\phi^1)\phi^2S(a\phi^3)=S(\phi^1a)\phi^2S(\phi^3)
\end{equation}
for all $a\in A$. Next, recalling that $\alpha=1$ and that $S(a)=\beta s(a)$ for all $a\in A$:
\begin{displaymath}
\begin{split}
s(\phi^1)\alpha \phi^2\beta s(\phi^3) & \stackrel{\phantom{(50)}}{=} s(\phi^1)\phi^2 S(\phi^3)\stackrel{\eqref{eq:s}}{=}\Phi^1S(\phi^1\Phi^2)\Phi^3\phi^2S(\phi^3) \\
 & \hspace{1.4pt}\stackrel{(\blacktriangle)}{=}\hspace{1.4pt} \Phi^1S(\phi^1\Phi^2)\phi^2S(\phi^3)\Phi^3 \stackrel{\eqref{eq:eq3.60}}{=} \Phi^1S(\phi^1)\phi^2S(\Phi^2\phi^3)\Phi^3 \\
 & \hspace{0.4pt}\stackrel{(\triangle)}{=}\hspace{0.4pt} \Phi^1S(\phi^1)\phi^2S(\phi^3)s(\Phi^2)\Phi^3 \stackrel{\eqref{eq:3.39}}{=} \Phi^1S(1)s(\Phi^2)\Phi^3 \\
 & \stackrel{\eqref{eq:3.51}}{=}\Phi^1S(\Phi^2)\Phi^3\stackrel{\eqref{eq:P3}}{=}1
\end{split}
\end{displaymath}
\end{itemize}
where in $(\triangle)$ we used: $S(ab)=\beta s(ab)=\beta s(b)s(a)=S(b)s(a)$ and in $(\blacktriangle)$ we used \eqref{eq:eq3.59} again, with $a=\phi^2S(\phi^3)$.
\end{proof}

\begin{corollary}\label{cor3.3.11}\emph{(Dual to \cite[Theorem 2.16]{AP2})}
Let $(A,m,u,\Delta,\varepsilon,\Phi,S)$ be a quasi-bialgebra with preantipode. If $A$ is commutative, then all the conclusions of Proposition \ref{prop3.2.30} hold for $A$. In particular, it is an ordinary Hopf algebra.
\end{corollary}

Let us recall now the less trivial result we referred to at the very beginning of this section. This theorem is due to Schauenburg and it states that, at least in the finite-dimensional case, the existence of a preantipode is equivalent to the existence of a quasi-antipode.

\begin{theorem}\label{thSch}\emph{(\cite[Theorem 3.1]{Sch1})}
Let $(A,m,u,\Delta,\varepsilon,\Phi)$ be a finite-dimensional quasi-bialgebra. The following are equivalent:
\begin{itemize}
\item[(1)] $A$ is a quasi-Hopf algebra,
\item[(2)] the adjunction $(F,G,\eta,\epsilon)$ is a category equivalence.
\end{itemize}
\end{theorem}

\emph{The interested reader may find the proof of this equivalence in \cite{Sch1}, here we fix our attention on the important steps.}

Unfortunately, this is a non-constructive result. From the invertibility of the component of the unit $\eta$ associated to the quasi-Hopf bimodules $A\hat{\otimes }A$ as defined in \eqref{eq:etacap}, and from the finiteness of $A$, one deduces the existence of an isomorphism of left $A$-modules $\varfun{\widetilde{\gamma}}{\overline{A\otimes A}}{A}$ by applying the Krull-Schmidt Theorem.

With this $\widetilde{\gamma}$, on the one hand one defines $\varfun{\gamma}{A}{A}$ by $\gamma(a)=\widetilde{\gamma}(\cl{1\otimes a})$ for all $a\in A$. On the other hand, one equips $A$ with a left $A\otimes A$-module structure given by
\begin{equation}\label{eq:action1}
(x\otimes y)\triangleright a:=\widetilde{\gamma}((x\otimes y)\cdot\widetilde{\gamma}^{-1}(a))
\end{equation}
for all $x,y,a\in A$. The action of the right tensor factor comes out to be of the form
\begin{equation}\label{eq:action2}
(1\otimes x)\triangleright a=as(x)
\end{equation}
for all $a,x\in A$, and for a certain antimultiplicative endomorphism $s$ of $A$. Afterwards, one defines the isomorphism of quasi-Hopf bimodules
\begin{equation}\label{eq:thetadef}
\theta:=(\widetilde{\gamma}\otimes A)\circ\hat{\eta}\,\colon\,A\otimes A\longrightarrow A\otimes A
\end{equation}
and one considers the elements $\beta:=\gamma(1)$ and $\alpha:=(A\otimes \varepsilon)\theta^{-1}(1\otimes 1)$. The conclusion of the proof consists in verifying that $(s,\alpha,\beta)$ is a quasi-antipode for $A$.

As we said, this is a non-constructive proof, because of the role played by the Krull-Schmidt Theorem. However, there are two interesting relations that one can derive from this proof \emph{a posteriori}:
\begin{equation}\label{eq:44}
\widetilde{\gamma}(\cl{a\otimes b})=a\gamma(b)=a\beta s(b)\qquad\mathrm{and}\qquad a_1\gamma(a_2)=\varepsilon(a)\gamma(1).
\end{equation}
Let us compare these with the following ones:
\begin{equation*}
\xi(\cl{a\otimes b})=aS(b)=a\beta s(b)\alpha\qquad\mathrm{and}\qquad a_1S(a_2)=\varepsilon(a)S(1)
\end{equation*}
(cf. \eqref{eq:xiS} and \eqref{eq:P1}). Since they look like the same, we hoped that it was possible to obtain a constructive version by means of the preantipode. Unfortunately again, $\xi$ is not invertible in general, not even in the finite-dimensional case. Hence, it cannot substitute $\widetilde{\gamma}$ in Schauenburg's proof. Nevertheless, the following result holds.

\begin{corollary}\label{cor3.3.22}
Let $(A,m,u,\Delta,\varepsilon,\Phi,S)$ be a quasi-bialgebra with preantipode. If $\xi$ as defined in \eqref{eq:xiS} is bijective, then $A$ is a quasi-Hopf algebra with quasi-antipode given by $\alpha=1$, $\beta=S(1)$ and, for all $a\in A$,
$$s(a)=\xi\left((1\otimes a)\cdot\xi^{-1}(1)\right)=1^1S(a1^2)$$
where $\cl{1^1\otimes 1^2}=\xi^{-1}(1)$.
\end{corollary}

\begin{proof}
In proving Schauenburg's result \ref{thSch}, the finiteness condition on $A$ is used just to find an isomorphism $\varfun{\widetilde{\gamma}}{\overline{A\otimes A}}{A}$. By hypothesis, we already have such an isomorphism:
$$\lfun{\xi}{\overline{A\otimes A}}{A}{\cl{a\otimes b}}{aS(b)}.$$
Hence, let us substitute this $\xi$ to $\widetilde{\gamma}$ into Schauenburg's proof. We get that $\gamma=S$ and $\beta=S(1)$. Moreover,
$$s(a)=(1\otimes a)\triangleright 1=\xi\left((1\otimes a)\cdot \xi^{-1}(1)\right)=\xi\left(\cl{1^1\otimes a1^2}\right)=1^1S(a1^2)$$
and, recalling \eqref{eq:thetadef},
\begin{displaymath}
\alpha=(A\otimes \varepsilon)\left(\theta^{-1}(1\otimes 1)\right) =(A\otimes \varepsilon)\left(\hat{\eta}_A^{-1}(\xi^{-1}(1)\otimes 1)\right)\stackrel{\eqref{eq:xi}}{=}\xi\left(\xi^{-1}(1)\right)=1
\end{displaymath}
as claimed. Now, one easily checks that the triple $(s,\alpha,\beta)$ is a quasi-antipode. Indeed, recalling that $\xi=\widetilde{\gamma}$ and $S=\gamma$, we find out that the following relation holds for all $a,b\in A$ in this context:
\begin{equation}\label{eq:49}
a\beta s(b)\stackrel{\eqref{eq:44}}{=}\xi(\cl{a\otimes b})=aS(b).
\end{equation}
Consequently, we can take advantage of this to verify that
\begin{equation*}
\begin{split}
s(a_1)\alpha a_2 & \stackrel{\phantom{(62)}}{=} s(a_1)a_2=1^1S(a_11^2)a_2\stackrel{\eqref{eq:P2}}{=}1^1S(1^2)\,\varepsilon(a) \\
 & \stackrel{\eqref{eq:49}}{=}\xi(\cl{1^1\otimes1^2})\,\varepsilon(a)=\xi(\xi^{-1}(1))\,\varepsilon(a)=\varepsilon(a)\alpha
\end{split}
\end{equation*}
and that
$$a_1\beta s(a_2)\stackrel{\eqref{eq:49}}{=}a_1S(a_2)\stackrel{\eqref{eq:P1}}{=}\varepsilon(a)S(1)=\varepsilon(a)\beta.$$
Therefore, \eqref{eq:QHA1} and \eqref{eq:QHA2} are valid. Moreover, \eqref{eq:QHA3} holds because:
$$\Phi^1\beta s(\Phi^2)\alpha\Phi^3\stackrel{\eqref{eq:49}}{=}\Phi^1S(\Phi^2)\Phi^3\stackrel{\eqref{eq:P3}}{=}1.$$
Finally, consider $\theta=(\xi\otimes A)\circ \hat{\eta}$ again and the following relations:
\begin{gather}
\theta^{-1}(1\otimes 1)=\hat{\eta}^{-1}(\xi^{-1}(1)\otimes 1)\stackrel{\eqref{eq:30}}{=}1^1S(\phi^11^2)\phi^2\otimes \phi^3=s(\phi^1)\phi^2\otimes \phi^3, \label{eq:50} \\
(A\otimes \varepsilon)\circ\theta=(A\otimes \varepsilon)\circ(\xi\otimes A)\circ \hat{\eta}=(\xi\otimes \K)\circ(\overline{A\otimes A}\otimes \varepsilon)\circ\hat{\eta}\stackrel{\eqref{eq:etacap}}{=}\xi\circ\pi. \label{eq:51}
\end{gather}
where $\pi:A\otimes A\rightarrow \overline{A\otimes A}$ is just the canonical projection. These make the final check easier, since
\begin{equation*}
\begin{split}
s(\phi^1)\alpha\phi^2\beta s(\phi^3) & \stackrel{\eqref{eq:44}}{=}\xi(\cl{s(\phi^1)\phi^2\otimes \phi^3})\stackrel{\eqref{eq:50}}{=}\xi(\cl{\theta^{-1}(1\otimes 1)}) \\
 & \stackrel{\eqref{eq:51}}{=}\left((A\otimes \varepsilon)\circ\theta\right)(\theta^{-1}(1\otimes 1))=1
\end{split}
\end{equation*}
and so even \eqref{eq:QHA4} holds.
\end{proof}

Corollary \ref{cor3.3.22} retrieves what it seems to be a limited family of quasi-bialgebras with preantipode for which it is possible to recover an explicit relation with the quasi-Hopf algebra structure (as the one we will study in Example \ref{ex:3.3.23}). Let us show briefly that it is actually a large class of quasi-Hopf algebras.

Let $(A,m,u,\Delta,\varepsilon,\Phi,s,\alpha,\beta)$ be a finite-dimensional quasi-Hopf algebra. Then we know, by Theorem \ref{th3.2.22}, that $A$ admits a preantipode $S(\cdot):=\beta s(\cdot)\alpha$ and so the Structure Theorem holds for the quasi-Hopf $A$-bimodules. Applying Schauenburg's result \ref{thSch} we get, a posteriori, a quasi-antipode $(s',\alpha',\beta')$ for $A$ such that the morphism $\widetilde{\gamma}(\cl{x\otimes y})=x\beta' s'(y)$ is invertible. By \cite[Proposition 1.1, p. 1425]{Dri}, a quasi-antipode is uniquely determined up to an invertible element. Hence, there exists an $u\in A$ invertible such that $(s,\alpha,\beta)$ and $(s',\alpha',\beta')$ are connected via $u$. In particular, if $\alpha$ is invertible, then also $\alpha'$ is invertible. By the way, note that $s'$, $\alpha'$ and $\beta'$ are not known to us, since they are obtained by $\widetilde{\gamma}$.

Next, assume that $\alpha$ is invertible in $A$. Hence
\begin{equation}\label{eq:52}
\lfun{\xi}{\overline{A\otimes A}}{A}{\cl{x\otimes y}}{xS(y)=x\beta's'(y)\alpha'=\widetilde{\gamma}(\cl{x\otimes y})\alpha'}
\end{equation}
is invertible with `explicit' inverse given by
\begin{equation}\label{eq:53}
\xi^{-1}(h):=\widetilde{\gamma}^{-1}\left(h(\alpha')^{-1}\right)=\widetilde{\gamma}^{-1}\left(h\alpha^{-1}u^{-1}\right).
\end{equation}
Thus we can apply Corollary \ref{cor3.3.22}. This implies that, if $\alpha$ is invertible, it is always possible to recover explicitly the quasi-antipode from the preantipode, at least theoretically. It is just a question of finding an explicit inverse to the map $\xi$, that we know to be invertible.

\begin{remark}
By \cite[Proposition 1.1, p. 1425]{Dri}, if $(s,\alpha,\beta)$ and $(s',\alpha',\beta')$ are two quasi-antipodes for a quasi-bialgebra $(A,m,u,\Delta,\varepsilon,\Phi)$, then there exists a unique invertible element $v\in A$ such that:
$$s'(a)=vs(a)v^{-1},\qquad \alpha'=v\alpha, \qquad \beta'=\beta v^{-1}.$$
If $(A,m,u,\Delta,\varepsilon,\Phi,s,\alpha,\beta)$ is a quasi-Hopf algebra with $\alpha$ invertible, then Corollary \ref{cor3.3.22} states that $\left(\xi((1\otimes *)\cdot \xi^{-1}(1)),1,\beta s(1)\alpha\right)$ is a quasi-antipode for $(A,m,u,\Delta,\varepsilon,\Phi)$.
Consequently we should have that, under these hypothesis, $v=\alpha^{-1}$. In fact:
$$v\alpha=\alpha'=1=\alpha^{-1}\alpha, \qquad \beta v^{-1}=\beta'=\beta s(1)\alpha=\beta \alpha$$
and finally, for all $h\in A$
\begin{displaymath}
\begin{split}
s'(h) & \stackrel{\phantom{(56)}}{=}\xi\left((1\otimes h)\cdot\xi^{-1}\left(1\right)\right)\stackrel{\eqref{eq:53}}{=}\xi\left((1\otimes h)\cdot\widetilde{\gamma}^{-1}\left(\alpha^{-1}v^{-1}\right)\right) \\
 & \stackrel{\eqref{eq:52}}{=} \widetilde{\gamma}\left((1\otimes h)\cdot\widetilde{\gamma}^{-1}\left(\alpha^{-1}v^{-1}\right)\right)v\alpha \stackrel{\eqref{eq:action1}}{=} ((1\otimes h)\triangleright \alpha^{-1}v^{-1}) v\alpha \\
 & \stackrel{\eqref{eq:action2}}{=} \alpha^{-1}v^{-1}s'(h)v\alpha = \alpha^{-1}s(h)\alpha.
\end{split}
\end{displaymath}
\end{remark}

Observe that in this framework fall all finite-dimensional Hopf algebras, the quasi-Hopf algebras $H(2)$, $H_\pm(8)$ and $H(32)$ of \cite{EG}, the twisted quantum doubles $D^{\omega}(G)$ introduced by Dijkgraaf, Pasquier and Roche (cf. \cite{DPR}, \cite[Section XV.5]{Ka}), the basic quasi-Hopf algebras $A(q)$ of \cite{G}.

In order to find interesting examples of the relation that intervenes between quasi-antipodes and preantipodes, one should look for a quasi-Hopf algebra that is finite-dimensional and such that $\alpha$ is not invertible. Unfortunately, it will not be enough to transform a quasi-Hopf algebra with trivial $\alpha$ (let us call $\alpha$ \emph{trivial} when it is invertible) via a gauge transformation $F$, as the following remark shows.

\begin{remark}
Let $(A,m,u,\Delta,\varepsilon,\Phi,s,\alpha,\beta)$ be a (finite-dimensional) quasi-Hopf algebra with $\alpha$ invertible. We have the preantipode $S(\cdot)=\beta s(\cdot)\alpha$ and the quasi-antipode $(\widehat{s},\widehat{\alpha},\widehat{\beta})$ obtained from $S$ by Corollary \ref{cor3.3.22}. Let $F\in A\otimes A$ be a gauge transformation on $A$ and consider the quasi-antipode $(s,\alpha_F,\beta_F)$ as defined in \cite[Remark 5, p. 1425]{Dri}. In general, $\alpha_F=s(f^1)\alpha f^2$ needs not to be invertible. Nevertheless, consider the preantipode $S_F(\cdot)=\beta_Fs(\cdot)\alpha_F$
and denote by $E=E^1\otimes E^2$ and $G=G^1\otimes G^2$ other two copies of $F$. We have that, for all $h\in A$:
\begin{gather*}
\widehat{s}(h)=1^1S(h1^2)=1^1\left(S_F\right)_{F^{-1}}(h1^2)\stackrel{\eqref{eq:gauge}}{=}1^1f^1S_F(F^1h1^2f^2)F^2, \\
\widehat{\alpha}_F=\widehat{s}(g^1)\widehat{\alpha}g^2=1^1f^1S_F(F^1g^11^2f^2)F^2g^2=1^1f^1S_F(1^2f^2), \\
\widehat{\beta}_F=G^1\,\widehat{\beta}\,\widehat{s}(G^2)=G^1S(1)\widehat{s}(G^2)=G^1e^1S_F(E^1e^2)E^21^1f^1S_F(F^1G^21^2f^2)F^2,
\end{gather*}
and this quasi-antipode $(\widehat{s},\widehat{\alpha}_F,\widehat{\beta}_F)$ on $A_F$ is written `explicitly' using just $F$, $S_F$ and $\xi^{-1}(1)$. Furthermore, it can be proved that it is connected to $(s,\alpha_F,\beta_F)$ via the invertible element $\alpha$.
\end{remark}

Let us conclude with an explicit example.

\begin{example}\label{ex:3.3.23}(\cite[Preliminaries 2.3]{EG})
Let $C_2=\left\langle g\right\rangle$ be the cyclic group of order 2 with generator $g$ and let $\K$ be a field of characteristic different from 2. Consider the group bialgebra $H(2):=\K C_2$.
Observe that $H(2)$ is  a two dimensional commutative algebra. Now, let us set
$p:=\frac{1}{2}(1-g)$
and let us introduce the non trivial reassociator
$$\Phi:=(1\otimes 1\otimes 1)-2(p\otimes p\otimes p).$$
In this way $H(2)$ becomes a quasi-Hopf algebra with quasi-antipode given by $\beta=1$, $\alpha=g$ and $s=\id_{H(2)}$.
Let us consider the associated preantipode:
$$\lfun{S}{H(2)}{H(2)}{x}{xg}.$$
As above consider the map $\xi$ of \eqref{eq:xiS}:
$$\lfun{\xi}{\overline{H(2)\otimes H(2)}}{H(2)}{\cl{x\otimes y}}{xyg}.$$
It is easy to see that it is bijective and we can exhibit an explicit inverse for $\xi$. Indeed, consider the function $$\lfun{\psi}{H(2)}{\overline{H(2)\otimes H(2)}}{x}{\cl{x\otimes g}}.$$
By composing with $\xi$ we find:
\begin{displaymath}
\xi\left(\psi(x)\right)=\xi\left(\cl{x\otimes g}\right)=xgg=x
\end{displaymath}
and since we know that $\xi$ is invertible, we have that $\psi=\xi^{-1}$.
Therefore, we can construct a quasi-antipode for $H(2)$ by virtue of Corollary \ref{cor3.3.22}. Since $$\cl{1^1\otimes 1^2}=\xi^{-1}(1)=\psi(1)=\cl{1\otimes g},$$
what we find is $\alpha=1$, $\beta=S(1)=g$ and $s(x)=S(xg)=xgg=x$ for all $x\in H(2)$. Finally, by recalling that a quasi-antipode in uniquely determined only up to an invertible element and that $g$ is trivially invertible, it is easy to see that $g$ itself is the invertible element that we need and thus we recovered the structure given previously.
\end{example}

\begin{remark}
Observe that, once we have proven that $H(2)$ is a quasi-bialgebra with preantipode $S$, we may come to the same conclusions of Example \ref{ex:3.3.23} by simply applying Corollary \ref{cor3.3.11}.
\end{remark} 

\smallskip

\textbf{Acknowledgements.}
I would like to thank Alessandro Ardizzoni for our discussions during the development of this work, for many helpful remarks and, most of all, for his careful reading of this paper. I would like to thank also Laiachi El Kaoutit, who helped me during the development of the master thesis, for his support and his encouragement. Moreover, my gratitude goes to the referee for useful suggestions.

\end{document}